\documentclass[12 pt]{amsart}
\usepackage{amssymb,latexsym,amsmath,amscd,amsthm,graphicx, color}
\usepackage[all]{xy}
\usepackage{pgf,tikz}
\usepackage{mathrsfs}
\usepackage{cite}
\usetikzlibrary{arrows}
\definecolor{uuuuuu}{rgb}{0.26666666666666666,0.26666666666666666,0.26666666666666666}
\definecolor{xdxdff}{rgb}{0.49019607843137253,0.49019607843137253,1.}
\definecolor{ffqqqq}{rgb}{1.,0.,0.}

\definecolor{uuuuuu}{rgb}{0.26666666666666666,0.26666666666666666,0.26666666666666666}
\definecolor{qqwuqq}{rgb}{0.,0.39215686274509803,0.}
\definecolor{zzttqq}{rgb}{0.6,0.2,0.}
\definecolor{xdxdff}{rgb}{0.49019607843137253,0.49019607843137253,1.}
\definecolor{qqqqff}{rgb}{0.,0.,1.}
\definecolor{cqcqcq}{rgb}{0.7529411764705882,0.7529411764705882,0.7529411764705882}

\theoremstyle{plain}

\newtheorem{theorem}[subsection]{Theorem}

\newtheorem{lemma}[subsection]{Lemma}

\newtheorem{prop}[subsection]{Proposition}
\theoremstyle{definition}
\newtheorem{defi}[subsection]{Definition}

\newtheorem{example}[subsection]{Example}

\newtheorem{remark}[subsection]{Remark}

\newtheorem{note}[subsection]{Note}

\newcommand{\uu}{\cup}% union
\newcommand{\ii}{\cap}% intersection
\newcommand{\sci}{\subset}% strictly contained in
\newcommand{\es}{\emptyset}% the empty set
\newcommand{\set}[1]{\{#1\}}% set
\newcommand{\ga}{\alpha}
\newcommand{\gb}{\beta}

\newcommand{\gd}{\delta}
\renewcommand{\gg}{\gamma}% old use >>

\newcommand{\tit}{\textit}% text italic

\newcommand{\B}{\boldsymbol}% Bold math symbol, use as \B{a}
\newcommand{\D}[1]{\mathbb{#1}}% Doubled - blackboard bold - only caps, uas as \D{A}
\newcommand{\te}{\text}% same as \mathrm command.

\begin{document}
To appear, Uniform Distribution Theory
\title{Optimal quantization for piecewise uniform distributions}

\author{Joseph Rosenblatt}
\address{Department of Mathematical Sciences \\
Indiana University-Purdue University Indianapolis\\
402 N. Blackford Street \\
Indianapolis, IN 46202-3217, USA.}
\email{rosnbltt@illinois.edu}

 \author{Mrinal Kanti Roychowdhury}
\address{School of Mathematical and Statistical Sciences\\
University of Texas Rio Grande Valley\\
1201 West University Drive\\
Edinburg, TX 78539-2999, USA.}
\email{mrinal.roychowdhury@utrgv.edu}

\date{}
\thanks{The research of the second author was supported by U.S. National Security Agency (NSA) Grant H98230-14-1-0320}

\date{}
\maketitle

\pagestyle{myheadings}\markboth{Joseph Rosenblatt and Mrinal Kanti Roychowdhury}{Optimal quantization for piecewise uniform distributions}

\begin{abstract}
Quantization for a probability distribution refers to the idea of estimating a given probability by a discrete probability supported by a finite number of points. In this paper, firstly a general approach to this process is outlined using independent random variables and ergodic maps; these give asymptotically the optimal sets of $n$-means and the $n$th quantization errors for all positive integers $n$.  Secondly two piecewise uniform distributions are considered on $\mathbb R$: one with infinite number of pieces and one with finite number of pieces. For these two probability measures, we describe the optimal sets of $n$-means and the $n$th quantization errors for all $n\in \mathbb N$. It is seen that for a uniform distribution with infinite number of pieces to determine the optimal sets of $n$-means for $n\geq 2$ one needs to know an optimal set of $(n-1)$-means, but for a uniform distribution with finite number of pieces one can directly determine the optimal sets of $n$-means and the $n$th quantization errors for all $n\in \mathbb N$.
\end{abstract}

\section{Introduction}

Quantization is the process of converting a continuous analog signal into
a digital signal of $k$ discrete levels, or converting a digital signal of $n$
levels into another digital signal of $k$ levels, where $k < n$. It is essential
when analog quantities are represented, processed, stored, or transmitted
by a digital system, or when data compression is required. It is a classic
and still very active research topic in source coding and information
theory. It has broad application in engineering and technology, for example in signal processing and data compression (see \cite{GG, GN, Z}). For mathematical treatment of quantization one is referred to Graf and Luschgy's book (see \cite{GL}). For most recent work on quantization for uniform distributions interested readers can see \cite{DR, R}. Let $P$ denote a Borel probability measure on $\D R^d$ and let $\|\cdot\|$ denote the Euclidean norm on $\D R^d$ for any $d\geq 1$.  Then, the $n$th \textit{quantization
error} for $P$ (of order $2$) is defined by
\begin{equation*} \label{eq1} V_n:=V_n(P)=\inf \Big\{\int \min_{a\in\alpha} \|x-a\|^2 dP(x) : \alpha \subset \mathbb R^d, \text{ card}(\alpha) \leq n \Big\},\end{equation*}
where the infimum is taken over all subsets $\ga$ of $\D R^d$ with card$(\ga)\leq n$ for $n\geq 1$. We assume that $\int\|x\|^2 dP(x)<\infty$ to make sure that there is a set $\ga$ for which the infimum occurs (see \cite{AW, GKL, GL, GL2}). Such a set $\ga$ for which the infimum occurs and contains no more than $n$-points is called an \tit{optimal set of $n$-means} and the elements of an optimal set are called \tit{optimal quantizers}. Let $U$ be the largest open subset of $\D R^d$ for which $P(U)=0$. Then, $\D R^d\setminus U$ is called the support of $P$, and is denoted by $\te{supp}(P)$. Notice that if $\te{supp}(P)$ is finite, i.e., if $\te{card}(\te{supp}(P))=N$ for some positive integer $N$, then $V_{n}(P)=0$ for all $n\geq N$. On the other hand, if the support of $P$ is countable, or if $P$ is a continuous probability measure, then an optimal set of $n$-means contains exactly $n$-elements, i.e., $V_n(P)>V_{n+1}(P)$ for all $n\in \D N$ (also see \cite{GL}).
For a finite set $\ga \sci \D R^d$, by $M(a|\ga)$ we denote the set of all elements in $\D R^d$ which are nearest to $a$ among all the elements in $\ga$, i.e.,
\[M(a|\ga)=\set{x \in \D R^d : \|x-a\|=\min_{b \in \ga}\|x-b\|}.\]
$M(a|\ga)$ is called the \tit{Voronoi region} generated by $a\in \ga$. On the other hand, the set $\set{M(a|\ga) : a \in \ga}$ is called the \tit{Voronoi diagram} or \tit{Voronoi tessellation} of $\D R^d$ with respect to the set $\ga$. Let us now state the following proposition (see \cite{GG, GL}).
\begin{prop} \label{prop0}
Let $\alpha$ be an optimal set of $n$-means with respect to a probability distribution $P$, $a \in \ga$, and $M (a|\ga)$ be the Voronoi region generated by $a\in \ga$. Then, for every $a \in\ga$,

$(i)$ $P(M(a|\ga))>0$, $(ii)$ $ P(\partial M(a|\ga))=0$, and $(iii)$ $a=E(X : X \in M(a|\ga))$.
\end{prop}

Notice that for $a\in \ga$, $a=E(X : X \in M(a|\ga))$ implies that the point $a$ is the conditional expectation of the random variable $X$ given that $X$ takes values in the Voronoi region $M(a|\ga)$. In \cite{DR}, Dettmann and Roychowdhury considered a uniform distribution on an equilateral triangle, and investigated the optimal sets of $n$-means and the $n$th quantization errors for the uniform distribution for all $n\geq 2$. In this direction one can also see \cite{R}. In this paper, in Section~\ref{sec2} we describe some general approaches to construct asymptotically optimal $n$-means that are highly worth considering, and it seems that they have not been looked at in the applied or theoretical literature on quantization.  Then, after some preliminaries in Section~\ref{sec3}, and in Section~\ref{sec4}, we analyze optimality for a piecewise uniform distribution with infinitely many pieces on the real line, and in Section~\ref{sec5}, we analyze optimality for a piecewise uniform distribution with finitely many pieces. For the uniform distribution with infinitely many pieces, in Lemma~\ref{lemma21} and Lemma~\ref{lemma22}, we first determine the optimal sets of $n$-means and the $n$th quantization errors for $n=2$ and $n=3$. Then, we prove Proposition~\ref{prop212}, Proposition~\ref{prop23}, Proposition~\ref{prop214} and  Proposition~\ref{prop215}, which help us to give the definition Definition~\ref{defi1} of a \tit{canonical sequence}. With the help of the canonical sequences, in Theorem~\ref{Th1}, we give an induction formula to determine the optimal sets of $n$-means and the $n$th quantization errors for all $n\geq 2$. We also give a tabular representation of several canonical sequences. For the uniform distribution with finitely many pieces, described in Section~\ref{sec5}, one can directly determine the optimal sets of $n$-means and the $n$th quantization error for any $n\in \D N$, induction formula is not needed in this case.

\section{The General Setting}\label{sec2}

We are interested in explicit sequences that are optimal $n$-means, or asymptotically optimal $n$-means, for given probability measures.  In later sections of this article, explicit $n$-means will be derived for piecewise uniform measures in a couple of different scenarios.  For now, as a way of framing issues with  and motivating that work, we want to consider some simple ways of generating discrete finite sets of points that can possibly be asymptotically optimal $n$-means, if not optimal ones, and get some control on the rate that the distortion error tends to zero.

The methods we consider here are both random models with uncorrelated variables and dynamical models in which there can be correlation of the outputs.  Each has advantages over the other.  They also have advantages over carrying out the detailed, hard work needed to construct
explicit optimal $n$-means with the trade-off being that one generally obtains only asymptotically optimal results.

For concreteness, we keep this introductory discussion limited to the interval $[0,1)\mod 1$ in Lebesgue measure.  We are interested in easy methods of obtaining a sequence $(\beta(k): k \ge 1)$ such that for all $n$,
$\int_0^1 \min\limits_{1 \le k \le n} |x - \beta(k)|^r\, dx$ is as small as possible.  The classical case is with $r =2$.   Indeed, it is also reasonable to consider the unaveraged error $\min\limits_{1 \le k \le n} |x - \beta(k)|$ itself.  Given a choice of $(\beta(k):k\ge 1)$, we would like to know the exact rate at which the distortion error tends to zero, and compare that with the optimal distortion error rate.

\subsection {IID Models}

Consider a method of randomly generating $n$-means for this simplest case of uniform measure on the interval $[0,1)$ modulo one.  We take $\B{\beta} = (\beta(k): k \ge 1)$ to be IID random variables with uniform distribution.  We actually are taking $\beta(k,\omega)$ with $\omega \in \Omega$ as the model underlying probability space $(\Omega,P)$, but we will suppress the dependence on $\omega$ if it will not create confusion.

The naive approach would be to estimate how many terms $(\beta(1),\dots,\beta(n))$ are needed so that each interval $I_j = [j/M,(j+1)/M),$ for $j = 0,\dots,M-1$, contains at least one point, with high probability.  This will guarantee that the quantization error $\int_0^1 \min\limits_{1 \le k \le n}|x - \beta(k)|^2 \, dx$ is no larger than $M\int_0^{1/M} x^2 \,dx = 1/3M^2$, a common estimate for the optimal quantization error.  It is easiest to consider the probability of the complementary case: there is some $I_j$ such that no term $\beta(k),k=1,\dots,n$ is in $I_j$.  This probability is $(1 - \frac 1M)^n$ for each such $j$.  So an estimate for the entire scope of the possibility is $M(1 - \frac 1M)^n$.  Taking $M = n/\ln(n)$ as a real variable would give for large $n$,
$M(1 - \frac 1M)^ n \sim 1/\ln(n)$.  Hence, with probability $1 - 1/\ln (n)$, each $I_j$ contains some $\beta(k), 1 \le k \le n$.  This gives the estimate $1/3M^2 = \ln^2(n)/2n^2 $ for the quantization error with this probability.  Asymptotically, this translates to taking $M \ge 1$ and then $n = M\ln (M)$ as a real variable to derive the same estimate with probability $1 - 1/\ln M\asymp 1 - 1/n$ as $n \to \infty$.  This only gives convergence in distribution as $n$ goes to $\infty$, but a simple increase in growth of $M$ can guarantee an almost sure result.  Note: instead of the optimal distortion error of $C/M^2\ln^2(M)$, this approach is giving a somewhat worse estimate $C/M^2$.

However, we can do better.  Consider the probability $P(\{\omega:n\min\limits_{1\le k \le n} |x - \beta(k,\omega)| \ge t\})$.  It is easy to see that this is $(1 - \frac {2t}n)^n$.  So scaling of the distortion error by $n$ results in convergence in distribution to the distribution function $d(t) = 1 - e^{-2t}, t \ge 0$, one can also compute expectations, and other moments.  For example,
\begin{align*}
&\int_{\Omega} n\min\limits_{1\le k \le n} |x - \beta(k,\omega)|\, dP(\omega)\\
& =\int_0^\infty P(\{\omega: n\min\limits_{1\le k \le n} |x - \beta(k,\omega)| \ge t\}) \, dt
= \int\limits_0^{n/2} (1  - 2t/n)^n \, dt
= \frac n{2(n+1)}.
\end{align*}
Going further than this distributional convergence is not going to be possible because of the Hewitt-Savage Theorem~\cite{HS}.  It shows that if this sequence converges a.e. or even just in measure, then the limit function would be a constant.  The distributional convergence shows that this is not possible.

But if we also integrate with respect to $x$ instead of $\omega$, then there is a.s. convergence to a computable constant.  That is, there is a non-zero constant $C$ such that for a.e. $\omega$, $\int_0^1 n\min\limits_{1\le k \le n} |x - \beta(k,\omega)| \, dx$ converges to $C$ as $n\to \infty$.  This is not a difficult calculation, if we use estimates of the series of variances for this distortion rate.  This convergence, indeed the distributional convergence above, shows that the random $n$-means are asymptotically optimal.  For details of the calculations in greater generality, see Cohort~\cite{PC}.  This article contains other interesting results related to a.s. convergence of the random proxy for optimal $n$-means and conclusions that follow about the asymptotic optimality of the random $n$-means.

The quantization process is closely related to the discrepancy estimates for the random sequence $(\beta(k,\omega))$.  See Kuipers and Niederreiter~\cite{KN}, especially the chapter notes, for a wealth of background information and references on discrepancy.  We again take our interval modulo one, but we suppress this in the notation for simplicity.

\begin{defi} Given a sequence $ \B{\beta} = (\beta(k):k\ge 1)$ in $[0,1)$, the discrepancy $D_n(\B{\beta})$ is defined by  \[D_n(\B{\beta}) =
\sup\left \{|\frac 1n\sum\limits_{k=1}^n 1_{[x,y)}(\beta(k)) - (y - x)|: 0 \le x < y < 1\right \}.\]
The smaller discrepancy  $D_n^*(\B{\beta})$ is defined by
\[D_n^*(\B{\beta}) =
\sup\left \{|\frac 1n\sum\limits_{k=1}^n 1_{[0,y)}(\beta(k)) - y|: 0 \le  y < 1\right \}.\]
\end{defi}

\noindent It is easy to see that $D_n^* \le D_n \le 2D_n^*$.

Now if $D_n < 1/M$, then for any interval $I$ of length $1/M$, there must be some $\beta_k \in I$ with $k \le n$.  So $\min\limits_{1 \le k \le n} |x - \beta(k)| \le 1/M$ too.  Hence, we have the useful basic estimate:
\begin{lemma}  $\min\limits_{1 \le k \le n} |x - \beta(k)| \le D_n(\B{\beta})$.
\end{lemma}
\noindent Thus, the following result of K-L Chung~\cite{C} gives an upper bound on the distortion error.

\begin{theorem} For a.e. $\omega$,
\[\limsup\limits_{n\to \infty} \frac
{\sqrt {2n}D_n^*(\B{\beta}(\omega))}{\sqrt {\ln\ln (n)}} = 1.\]
\end{theorem}

\noindent However, the actual distortion error rate here is likely to be faster.   That is, if we take $d_n(\B{\beta}(\omega)) = \min\limits_{1 \le k \le n} |x - \beta(k,\omega)|$, then some experimentation with estimates suggested that $\limsup\limits_{n\to \infty} \frac
{nd_n(\B{\beta}(\omega)) }{\ln n} < \infty$ for a.e. $\omega$.  Indeed, this is the case.  It was perhaps first proved by L\'evy~\cite{L}.  But many sophisticated extension of this have been achieved, many under the title or order statistics.  See for example the article by Deheuvels~\cite{D}.

If the measure that we are quantizing is not uniform, then we need to adjust the placement of the random variables $(\beta(k): k \ge 1)$.  The obvious approach is to just take $\beta(k)$ to be IID with distribution given by  the fixed probability measure $\nu$.  Notice that then we would under some general assumptions have the empirical measures $\frac 1n \sum\limits_{k=1}^n \delta_{\beta(k)}$ converging weakly to $\nu$.  The result of Theorem 7.5 in Graf and Luschgy~\cite{GL} shows that our random empirical measure would not be asymptotically  optimal except in the case of uniform measure.   However, given an absolutely continuous measure $d\nu = h d\lambda$, with a regular density function $h$, we could choose the $\beta(k)$ to be distributed according to the law $h^3 d\lambda$.  Then we would not only get a good estimate for the quantization error, but we would also have the empirical measures converging weakly to $h d\lambda = d\nu$ itself.  See Graf and Luschgy~\cite {GL} discussion following Theorem 7.5.

\subsection{Ergodic and Diophantine Models}
Consider a dynamical systems approach to asymptotically optimal $n$-means.  For this model, we take an ergodic, measure-preserving mapping $\tau$ of $[0,1]\mod 1$.  For a fixed $y \in [0,1]$, let $\beta(k,y) = \tau^k(y)$.  What can we say about the rate that
$\min\limits_{1\le k \le n} |x - \beta(k,y)|$ tends to zero for arbitrary $x$, and at least a.e. $y$?  Also, is there better stabilization of this if we instead consider the mean behavior
$\int_0^1 \min\limits_{1\le k \le n} |x - \beta(k,y)|^2 \, dx$?  This is the stationary version of the IID case above, where correlation of the $n$-means is being allowed.

So far we know some things, but not enough about this variation on possible asymptotically optimal $n$-means.  Results in this direction will appear in future work.  But it is clear that the ergodicity is not needed for the most important property in obtaining asymptotically optimal $n$-means.  What ergodicity implies is that for a.e. $y$, the orbit  $(\tau^k(y): k \ge 1)$ is dense in $[0,1]$.  This is all that is needed for
$\min\limits_{1\le k \le n} |x - \beta(k,y)|$ to converge to zero for $x$.
What then happens if instead we take as our map a minimal map of $[0,1]$?  The same property would hold for all points.  That is, if we have a minimal map $\tau$ of a compact, metric space $(X,d_X)$, in place of $[0,1]$, then $\min\limits_{1\le k \le n} d_X(x,\tau^k(y))$ also tends to zero for arbitrary $x$ and $y$.  In any such case, it is in general not clear how to obtain a rate for the distortion error, or specific information about the distribution of the $n$-means that are resulting.  This type of issue is why the specific details presented in this article in Section~\ref{sec4} and Section~\ref{sec5} are so useful.  Concrete, completely described optimal $n$-means are worth a great deal in any applied, or theoretical, quantization process.

We might also consider a relative of the dynamical systems approach: a Diophantine method.  Now we take $\beta(k,\theta) = \{k\theta\}$ for all  $k \ge 1$, where $\theta$ is some irrational number and $\{t\}$ denotes the fraction in $[0,1)$ such that $t = \{t\} + k$ for some integer $k$.  We know that $\B{\beta}(\theta) = (\beta (k,\theta): k \ge 1)$ is uniformly distributed in $[0,1]$ and moreover there is an estimate on the discrepancy $D_n(\B{\beta}(\theta))$ that holds for a.e. $\theta$ that comes from classical facts about continued fractions and
Diophantine approximation.  The estimate gives for a.e. $\theta$ and  for all $\delta > 0$, $D_n(\B{\beta}(\theta)) \le \ln((n)^{1 + \delta}/n$ for large enough $n$.  But then if $D_n(\B{\beta}(\theta)) < \frac 1M$, we must have for any interval $I \subset [0,1]$ with $|I|= \frac 1M$, there is  some $k\theta \in I$ with $1 \le k \le n$.  This then gives the discrete set $\{\beta(k) : 1 \le k \leq  n\}$ with a quantization error no larger than $1/3M^2$.  Again, we can translate this to real values by taking $n = M\ln^{1+\delta} (M)$ asymptotically to achieve this quantization error $C/M^2$.  It is not as good as the optimal one that would be $C/M^2\ln^{2+2\delta}(M)$.  Despite the fact that the discrepancy estimate here is better than for the one in the IID case, the unaveraged distortion error is not as good as what one can obtain in the IID case.  The virtue of the Diophantine result is that it is explicit.

What we are observing is that the same approach to over-estimating the distortion error that was used in the random approach will work for this Diophantine approach, replacing the iterated logarithm method of Chung with the theorem of Khinchin~\cite{K}.  See also Kuipers and Niederreiter~\cite{KN} again.  To be more exact, Khinchin's theorem says for any non-decreasing $g$ such that $\sum\limits_{n=1}^\infty \frac 1{g(n)} < \infty$, for a.e. $\theta$, one has for the sequence ${\B{\beta}(\theta)} = (k\theta \mod 1: k \ge 1)$
	\[nD_n(\B{\beta}(\theta)) = O(\ln (n) g(\ln \ln (n))).\]
But just as it proved to be the case in the IID model, using discrepancy for the Diophantine model, to over estimate the Diophantine model distortion error, seems likely to give too large an estimate.  For example, see the results in Graham and Van Lint~\cite{GVL}.  This article not only shows that there is a necessary spread in the distortion rate, but it shows that the optimal behavior for the Diophantine model is with $\theta$ that have bounded terms in the simple continued fraction expansion.  For these, the distortion error is on the order of the optimal distortion error i.e. $d_n(\theta) = O(1/n)$. What is not shown in \cite{GVL}, and seems missing in the literature, is a metric result that gives optimal control on the distortion rate for a.e. $\theta$.

So it is possible that the dynamical system result or the Diophantine result can be improved by a couple of different approaches. One approach is to not consider the random input value, but take a specific very good value of $\theta$, actually the Golden Mean.  As mentioned above, this is what is considered in Graham and Van Lint~\cite{GVL}.  See also Motta, Shipman, and Springer~\cite{MSS} where optimal transitivity is studied to limit the gaps in the sequence.  Another approach would be to use bounded remainder sets so that the discrepancy error can be perhaps better controlled.  See both Haynes, Kelly, and Koivusalo~\cite{HKK}; and Haynes and Koivusalo~\cite{HK}.

 In addition,  we conjecture the following relationships between the asymptotic results from dynamical models and the optimal results that follow in later sections of this paper.  Indeed, let $(\beta(k):1 \le k \le n)$  be either
the dynamical system or Diophantine construction above.  Let $(\alpha_n(k):1 \le k \le n)$ be an optimal set of $n$-means.  While the unaveraged distortion rate is not going to be as good as the optimal distortion rate, averaging seems to have a very strong impact (as is shown in the IID case by Cohort~\cite{C}).  We conjecture though that for every constant $K$, when $n$ is sufficiently large,
\[K+\int_0^1 \min\limits_{1 \le k \le n} |x - \alpha_n(k)|^2 \, dx \le \int_0^1 \min\limits_{1 \le k \le n} |x - \beta(k)|^2\, dx.\]
This result would show that the optimal $n$-means are certainly better than either the random or dynamical approach to quantization.  On the other hand, we also see that there may be lots of examples such that for every constant $R > 1$, when $n$ is sufficiently large,
\[R\int_0^1 \min\limits_{1 \le k \le n} |x - \alpha_n(k)|^2 \, dx \ge \int_0^1 \min\limits_{1 \le k \le n} |x - \beta(k)|^2\, dx.\]
This would mean that the optimal $n$-means are not better as far as the asymptotic behavior of the associated distortion rates are concerned, and that the random or dynamical system approaches give asymptotically optimal $n$-means.

 We summarize what has been demonstrated in this section, Section~\ref{sec2}.  Both the random and the dynamical approaches to quantization give fairly good quantization, but as we will see they do not give as good a quantization error as is possible using optimal quantization.  This fact alone should help to motivate why we want to have explicitly optimal $n$-means.  To accomplish this, in the later sections of this paper we take some care to describe completely how to get optimal $n$-means in a number of different contexts.

\section{Notation and Some Facts} \label{sec3}

Let $P$ be a piecewise uniform distribution with infinitely many pieces on the real line with probability density function (pdf) $f$ given by
\[f(x)=\left\{\begin{array}{ccc}
(\frac 32)^n & \te{ if } 1-\frac 1{3^{n-1}}\leq x\leq 1- \frac 2{3^n} \te{ for } n\in \D N,\\
 0  & \te{ otherwise}.
\end{array}\right.
\]
In the sequel we will write $J_n:=[1-\frac{1}{3^{n-1}}, 1-\frac{2}{3^n}]$ and $J_{(n, \infty)}:=\mathop{\uu}\limits_{j=n+1}^\infty J_j$, where $n\in  \D N$. For $n\in \D N$, by $J_n(0)$ and $J_n(1)$, we denote the left and right end points of the interval $J_n$, respectively, i.e., $J_n(0)=1-\frac{1}{3^{n-1}}$ and $J_n(1)=1-\frac{2}{3^n}$.
\begin{lemma}\label{lemma1}
Let $E(P)$ and $V(P)$ represent the expected value and the variance of a random variable $X$ with distribution $P$. Then, $E(P)=\frac 12$ and $V(P)=\frac{25}{204}$.
\end{lemma}

\begin{proof} We have
\[E(P)=\sum _{n=1}^{\infty } \int_{J_n}  x dP=\frac 12, \te{ and } V(P)=\sum _{n=1}^{\infty }\int_{J_n}(x-\frac 12)^2 dP=\frac{25}{204},\]
and thus the lemma is yielded.
\end{proof}

\begin{note} Lemma~\ref{lemma1} implies that the optimal set of one-mean is $\set{\frac 1 2}$ and the corresponding quantization error is $\frac{25}{204}$.
Let $k\in \D N$. By $P(\cdot|J_k)$ we denote the restriction of the probability measure $P$ on the interval $J_k$, i.e., $P(\cdot|J_k)=P(\cdot\ii J_k)/P(J_k)$, in other words, for any Borel subset $B$ of $J_k$ we have $P(B|J_k)=\frac{P(B\ii J_k)}{P(J_k)}$. Similarly, write $P(\cdot|J_{(k,\infty)})$ to denote the restriction of the probability measure $P$ on $J_{(k, \infty)}$. For a probability distribution $Q$, by $\ga_n(Q)$, we denote an optimal set of $n$-means for $Q$. For a Borel subset $B$ of $\D R$, by $V(P, \ga_n(Q), B)$, it is meant the quantization error (or distortion measure) contributed by $\ga_n(Q)$ on the set $B$ with respect to the probability distribution $P$. If nothing is mentioned within a parenthesis, by $\ga_n$ and $V_n$, it is meant an optimal set of $n$-means and the $n$th quantization error with respect to the probability distribution $P$.
\end{note}

\begin{lemma} \label{lemma2}
For $k\in \D N$, let $E(P(\cdot|J_k))$ and $E(P(\cdot|J_{(k,\infty)}))$ denote the expectations of the random variables with distributions $P(\cdot|J_k)$ and $P(\cdot|J_{(k, \infty)})$, respectively. Then,
\[E(P(\cdot|J_k))=1-\frac 52 \frac 1{3^k} \te{ and } E(P(\cdot|J_{(k,\infty)}))=1-\frac 12 \frac 1{3^k}.\]
\end{lemma}

\begin{proof} By the definition of the conditional expectation, we have
\[E(P(\cdot|J_k))=\int_{J_k} x dP(\cdot|J_k)=\frac 1{P(J_k)} \int_{J_k}x dP =2^k\int_{J_k} (\frac 3 2)^k x dx=1-\frac 52 \frac 1{3^k}, \te{ and } \]
\begin{align*}  E(P(\cdot|J_{(k,\infty)}))&=\int_{J_{(k, \infty)}}x d P(\cdot|J_{(k,\infty)})=\frac 1{P(J_{(k,\infty)})} \sum_{j={k+1}}^\infty \int_{J_j} x dP=2^k\sum_{j={k+1}}^\infty \int_{J_j} (\frac 32)^j x dx,
\end{align*}
implying $E(P(\cdot|J_{(k,\infty)}))=1-\frac 12 \frac 1{3^k}$,
and thus the lemma is yielded.
\end{proof}

\begin{remark} Lemma~\ref{lemma2} implies that $\ga_1(P(\cdot|J_k))=\set{1-\frac 52 \frac 1{3^k}}$,  $\ga_1(P(\cdot|J_{(k,\infty)}))=\set{1-\frac 12 \frac 1{3^k}}$, $E(P(\cdot|J_k))=\frac 12(J_k(0)+J_k(1))$,  and $E(P(\cdot|J_{(k,\infty)}))=\frac 12 (J_{k+1}(1)+J_{k+2}(0))$. $E(P(\cdot|J_{(k,\infty)}))$ can also be calculated in the following way:
\[E(P(\cdot|J_{(k,\infty)}))=\frac 1{P(J_{(k,\infty)})} \sum_{j={k+1}}^\infty P(J_j)E(P(\cdot|J_j))=2^k \sum_{j={k+1}}^\infty \frac 1{2^j}(1-\frac 52 \frac 1{3^j})=1-\frac 12 \frac 1{3^k}.\]
\end{remark}

\begin{prop}\label{prop1}
Let $k, n\in \D N$. Then, the set $\set{1-\frac{1}{3^{k-1}}+\frac{2i-1}{2n}\frac 1{3^k} : 1\leq i\leq n}$ is a unique optimal set of $n$-means for $P(\cdot|J_k)$, i.e., $\ga_n(P(\cdot|J_k))=\set{1-\frac{1}{3^{k-1}}+\frac{2i-1}{2n}\frac 1{3^k} : 1\leq i\leq n}$. Moreover,
\[V(P, \ga_n(P(\cdot|J_k)), J_k)=\frac 1{n^2}\frac 1{12} \frac{1}{18^k} \te{ and } V(P, \ga_1(P(\cdot|J_{(k, \infty)})), J_{(k, \infty)})=\frac{25}{204}\frac 1{18^k}.\]
\end{prop}
\begin{proof}
Since $P(\cdot|J_k)$ is uniformly distributed on $J_k$, the boundaries of the Voronoi regions of an optimal set of $n$-means will divide the interval $[1-\frac{1}{3^{k-1}}, 1-\frac 2{3^k}]$ into $n$ equal subintervals, i.e., the boundaries of the Voronoi regions are given by
\[\set{1-\frac{1}{3^{k-1}}, \, 1-\frac{1}{3^{k-1}}+\frac {1}{n}\frac 1{3^k}, \, 1-\frac{1}{3^{k-1}}+\frac {2}{n}\frac 1{3^k},\, \cdots, 1-\frac{1}{3^{k-1}}+\frac {n-1}{n}\frac 1{3^k}, 1-\frac 2{3^k}}.\]
This implies that an optimal set of $n$-means for $P(\cdot|J_k)$ is unique, and it consists of the midpoints of the boundaries of the Voronoi regions, i.e., the optimal set of $n$-means for $P(\cdot|J_k)$ is given by $\set{1-\frac{1}{3^{k-1}}+\frac{2i-1}{2n}\frac 1{3^k} : 1\leq i\leq n}$ for any $n\geq 1$. Then, the $n$th quantization error for $P$ due to the set $\ga_n(P(\cdot|J_k))$ on $J_k$ is given by
\begin{align*}
&V(P, \ga_n(P(\cdot|J_k)), J_k)=n \times (\te{the quantization error in each Voronoi region})\\
&=n \Big(\int_{[1-\frac{1}{3^{k-1}}, 1-\frac{1}{3^{k-1}}+\frac{1}{n}\frac 1{3^k}]} \Big(\frac 32\Big)^k \Big(x-(1-\frac{1}{3^{k-1}}+\frac{1}{2n}\frac 1{3^k})\Big)^2  dx\Big),
\end{align*}
which after simplification implies $V(P, \ga_n(P(\cdot|J_k)), J_k)= \frac 1{n^2}\frac 1{12} \frac{1}{18^k}$. Again, $E(P(\cdot|J_{(k,\infty)}))=1-\frac 12 \frac 1{3^k}$, and so,
 \[V(P, \ga_1(P(\cdot|J_{(k, \infty)})), J_{(k, \infty)})=\sum _{n=k+1}^{\infty }\int_{J_n} (x-(1-\frac 12 \frac 1{3^k}))^2 dP=\sum _{n=k+1}^{\infty }\int_{J_n} (\frac 32)^n(x-(1-\frac 12 \frac 1{3^k}))^2 dP,\]
 which upon simplification yields $V(P, \ga_1(P(\cdot|J_{(k, \infty)})), J_{(k, \infty)})=\frac{25}{204}\frac 1{18^k}$.
Thus, the proof of the proposition is complete.
\end{proof}
In the following section, we investigate the optimal sets of $n$-means for $n\geq 2$. Once the optimal sets of $n$-means are known the corresponding quantization error can easily be calculated.

\section{Optimal Sets of $n$-Means for $n\geq 2$} \label{sec4}
In this section, we first determine the optimal sets of $n$-means for $n=2$ and $n=3$.
\begin{lemma} \label{lemma21}
Let $\ga:=\set{a_1, a_2}$ be an optimal set of two-means such that $a_1<a_2$. Then, $a_1= \frac 16$ and $a_2=\frac 56$, and the corresponding quantization error is $V_2=\frac{7}{612}$.
\end{lemma}
\begin{proof} Consider the set of two points $\gb:=\set{\frac 16, \frac 56}$. The distortion error due to the set $\gb$ is given by
\[\int \min_{a \in \gb}(x-a)^2 dP=\int_{J_1}(x-\frac 16)^2 dP+\sum_{n=2}^\infty \int_{J_n} (x-\frac{5}{6})^2  dP=\frac{7}{612}.\]
Since $V_2$ is the quantization error for two-means, we have $V_2\leq \frac{7}{612}=0.0114379$. Let $\ga:=\set{a_1, a_2}$ be an optimal set of two-means such that $a_1<a_2$. Since the optimal quantizers are the expected values of their own Voronoi regions, we have $0<a_1<a_2<1$. If $\frac 13\leq a_1$, then
\[V_2\geq \int_{J_1}(x-\frac 13)^2 dP=\frac{1}{54}=0.0185185>V_2,\]
which leads to a contradiction. So, we can assume that $a_1< \frac 13$. If $a_2<\frac 23$, then
\[V_2\geq \sum_{n=2}^\infty \int_{J_n} (x-\frac 23)^2  dP=\frac{19}{918}=0.0206972>V_2,\]
which leads to another contradiction. So, we can assume that $\frac 23<a_2$. Since $0<a_1<\frac 13$ and $\frac 23<a_2<1$, we have $\frac 13<\frac 12(a_1+a_2)<\frac 23$ yielding the fact that the Voronoi region of $a_1$ does not contain any point from $J_{(1, \infty)}$ and the Voronoi region $a_2$ does not contain any point from $J_1$. This implies that $a_1=E(X : X\in J_1)=\frac 16$ and $a_2=E(X : X\in J_{(1, \infty)})=\frac 56$, and the corresponding quantization error is $V_2=\frac{7}{612}$, which is the lemma.
\end{proof}

\begin{lemma} \label{lemma22}
Let $\ga:=\set{a_1, a_2, a_3}$ be an optimal set of three-means such that $a_1<a_2<a_3$. Then, $a_1= \frac 16$, $a_2=\frac{13}{18}$, $a_3=\frac{17}{18}$, and  the corresponding quantization error is $V_3=\frac{29}{5508}$.
\end{lemma}

\begin{proof} Consider the set of three points $\gb:=\set{\frac 1 6, \frac{13}{18}, \frac{17}{18}}$. The distortion error due to the set $\gb$ is given by
\begin{equation} \label{eq45} \int_{J_1} (x-\frac{1}{6})^2 dP+\int_{J_2}(x-\frac{13}{18})^2 dP+\int_{J_{(2, \infty)}}  (x-\frac{17}{18})^2 dP=\frac{29}{5508}=0.00526507.\end{equation}
Since $V_3$ is the quantization error for three-means, we have $V_3\leq 0.00526507$. Let $\ga:=\set{a_1<a_2<a_3}$ be an optimal set of three-means. Since the optimal quantizers are the expected values of their own Voronoi regions we have $0<a_1<a_2<a_3<1$. If $\frac 13\leq a_1$, then
\[V_3\geq \int_{J_1}(x-\frac 13)^2 dP=\frac{1}{54}=0.0185185>V_3,\]
which leads to a contradiction. So, we can assume that $a_1< \frac 13$, and then the Voronoi region of $a_1$ does not contain any point from $J_{(1, \infty)}$. If it does, then we must have $\frac 12(a_1+a_2)>\frac 23$ implying $a_2>\frac 43-a_1\geq \frac 43-\frac 13=1$, which gives a contradiction. Thus, we see that $a_1\leq  E(X : X\in J_1)=\frac 16$.  Suppose that $a_2<\frac 12$. The following two cases can arise:

Case~1. Voronoi region of $a_2$ contains points from $J_{(1, \infty)}$.

Then, $\frac 12(a_2+a_3)>\frac 23$ implying $a_3>\frac 43-a_2\geq \frac 43-\frac 12=\frac 56$. First, assume that $\frac 56<a_3\leq \frac{31}{36}<J_3(0)$, and then
\[V_3\geq \int_{J_2} (x-\frac{5}{6})^2 dP+\sum _{n=3}^{\infty } \int_{J_{(2, \infty)}}(x-\frac{31}{36})^2 dP=\frac{481}{88128}=0.00545797>V_3,\]
which is a contradiction. Next, assume that $\frac {31}{36}\leq a_3$. Then, $\frac 12(\frac 23+\frac{31}{36})=\frac{55}{72}$. Also, notice that $E(X : X \in J_{(2, \infty)})=\frac {17}{18}$, and so,  we have
\begin{align*}
V_3&\geq \int_{[\frac 23, \frac{55}{72}]} (x-\frac{1}{2})^2 dP+\int_{[\frac{55}{72}, \frac 79]} (x-\frac{31}{36})^2 dP+\sum _{n=3}^{\infty } \int_{J_{(2, \infty)}}(x-\frac{17}{18})^2 dP\\
&=\frac{15431}{1410048}=0.0109436>V_3,
\end{align*}
which leads to a contradiction.

Case~2. Voronoi region of $a_2$ does not contain any point from $J_{(1, \infty)}$.

Then, as $E(X : X\in J_{(1, \infty)})=\frac 56$, we have
\[V_3\geq \int_{J_{(1, \infty)}}(x-\frac 56)^2 dP=\frac{25}{3672}=0.00680828>V_3,\]
which yields a contradiction.

Thus, by Case~1 and Case~2, we can assume that $\frac 12\leq a_2$. We now show that $P$-almost surely the Voronoi region of $a_2$ does not contain any point from $J_1$. For the sake of contradiction assume that the Voronoi region of $a_2$ contains points from $J_1$. Then, the distortion error contributed by $a_1$ and $a_2$ on the set $J_1$ is given by
\begin{align*}
&\int_{[0, \frac {a_1+a_2}{2}]}(x-a_1)^2 dP+\int_{[\frac{a_1+a_2}{2}, \frac 13]}(x-a_2)^2 dP=\frac{3 a_1^3}{8}+\frac{3}{8} a_2 a_1^2-\frac{3}{8} a_2^2 a_1+\frac{a_2^2}{2}-\frac{a_2}{6}-\frac{3 a_2^3}{8}+\frac{1}{54},
\end{align*}
which is minimum when $a_1=\frac 16$ and $a_2=\frac 12$. Then, notice that $\frac 12(a_1+a_2)=\frac 13$, i.e., $P$-almost surely the Voronoi region of $a_2$ does not contain any point from $J_1$. This implies the fact that $a_1=E(X : X \in J_1)=\frac 16$ and $\frac23 \leq  a_2$. Suppose that $\frac 79\leq  a_2$. Then,
\[V_3\geq \int_{J_1}(x-\frac 16)^2 dP+\int_{J_2}(x-\frac 79)^2 dP=\frac{11}{1944}=0.00565844>V_3,\]
which is a contradiction. So, we can assume that $\frac 23\leq  a_2< \frac 79$. Then, the Voronoi region of $a_2$ does not contain any point from $J_{(2, \infty)}$. If it does, then we must have $\frac 12(a_2+a_3)>\frac 89$ implying $a_3>\frac {16}9-a_2\geq \frac{16}{9}-\frac 79=1$, which yields a contradiction as $a_3<1$. Thus, we have $a_2=E(X : X \in J_2)=\frac {13}{18}$ and $a_3=E(X : X \in J_{(2, \infty)})=\frac {17}{18}$. Moreover, we have seen $a_1=\frac 16$. Then, by \eqref{eq45}, the quantization error is $V_3=\frac{29}{5508}$. This completes the proof of the lemma.
\end{proof}

\begin{prop} \label{prop212}
Let $n\geq 2$ and let $\ga_n$ be an optimal set of $n$-means. Then,

 $(i)$ $\ga_n\ii J_1\neq \es$ and  $\ga_n\ii [J_2(0), 1]\neq\es$;

 $(ii)$ $\ga_n$ does not contain any point from the open interval $(J_1(1), J_2(0)))$;

 $(iii)$ the Voronoi region of any point in $\ga_n\ii J_1$ does not contain any point from $[J_2(0), 1]$, and the Voronoi region of any point in $\ga_n\ii [J_2(0), 1]$ does not contain any point from $J_1$.
\end{prop}

\begin{proof}
By Lemma~\ref{lemma21} and Lemma~\ref{lemma22}, the proposition is true for $n=2, 3$. We now show that the proposition is true for all $n\geq 4$. Consider the set of four points $\gb:=\set{\frac 1{12}, \frac 14, \frac {13}{18}, \frac{17}{18}}$. The distortion error due to the set $\gb$ is given by
\begin{align*}
&\int\min_{a\in \gb}(x-a)^2 dP\\
&=\int_{[1, \frac 16]}(x-\frac 1{12})^2 dP+\int_{[\frac 16,\frac 13]}(x-\frac 14)^2 dP+\int_{J_2}(x-\frac {13}{18})^2 dP+\sum_{j=3}^\infty \int_{J_j}(x-\frac{17}{18})^2 dP=\frac{79}{44064}.
\end{align*}
Since $V_n$ is the quantization error for $n$-means with $n\geq 4$, we have $V_n\leq V_4\leq \frac{79}{44064}=0.00179285$. Let $\ga_n:=\set {0<a_1<a_2<\cdots<a_n<1}$ be an optimal set of $n$-means. If $\frac 13<a_1$, then
$V_n\geq \int_{J_1}(x-\frac 13)^2 dP=\frac 1{54}=0.0185185>V_n,$
which is a contradiction. If $a_n<J_2(0)=\frac 23$, then
\[V_n\geq \sum_{j=2}^\infty \int_{J_j}(x-\frac 23)^2 dP=\frac{19}{918}=0.0206972>V_n,\]
which leads to another contradiction. Thus, $\ga_n\ii J_1\neq \es$ and $\ga_n\ii [J_2(0), 1]\neq\es$, which completes the proof of $(i)$.

To prove $(ii)$ and $(iii)$, let $j:=\max\set{i : a_i\leq \frac 13}$. Then, $a_j\leq \frac 13$. We need to show that $\frac 23\leq a_{j+1}$. For the sake of contradiction, assume that $\frac 13<a_{j+1}<\frac 23$. If $\frac 13<a_{j+1}\leq \frac 12$, then $\frac 12(a_{j+1}+a_{j+2})>\frac 23$ implying $a_{j+2}>\frac 43-a_{j+1}\geq \frac 43-\frac 12=\frac 56>\frac 79$ and so,
$V_n\geq \int_{J_2}(x-\frac 56)^2 dP=\frac{13}{3888}=0.00334362>V_n,$
which yields a contradiction. Next, suppose that $\frac 12\leq a_{j+1}<\frac 23$. Then, $\frac 12(a_{j}+a_{j+1})<\frac 13$ implying $a_j<\frac 23 -a_{j+1}\leq \frac 23-\frac 12=\frac 16$, and so,
$V_n\geq \int_{[\frac 16, \frac 13]}(x-\frac 1 6)^2 dP=\frac{1}{432}=0.00231481>V_n,$
which gives a contradiction. So, we can assume that $a_j\leq \frac 13 <\frac 23\leq a_{j+1}$, i.e., $\ga_n$ does not contain any point from the open interval $(J_1(1), J_2(0))$, which yields $(ii)$.

If the Voronoi region of $a_j$ contains points from $[J_2(0), 1]$, we must have $\frac 12(a_j+a_{j+1})>\frac 23$ implying $a_{j+1}\geq \frac 43-a_j=\frac 43-\frac 13=1$, which is a contradiction. Similarly, if the Voronoi region of any point in $\ga_n\ii [J_2(0), 1]$ contains points from $J_1$, we will arrive at a contradiction. Thus, $(iii)$ is yielded, and this completes the proof of the proposition.
\end{proof}

\begin{prop} \label{prop23}
Let $\ga_n$ be an optimal set of $n$-means for $n\geq 4$. Then, $\te{card}(\ga_n\ii J_1)\geq 2$ and $\te{card}(\ga_n\ii [J_2(0), 1])\geq 2$.
\end{prop}
\begin{proof} As shown in the proof of Proposition~\ref{prop212}, since $V_n$ is the quantization error for $n$-means for $n\geq 4$, we have $V_n\leq V_4\leq \frac{79}{44064}=0.00179285$. By Proposition~\ref{prop212}, we have $\te{card}(\ga_n\ii J_1)\geq 1$ and $\te{card}(\ga_n\ii [J_2(0), 1])\geq 1$. First, we show that $\te{card}(\ga_n\ii [J_2(0), 1])\geq 2$. Suppose that $\te{card}(\ga_n\ii [J_2(0), 1])=1$. Then, as $E(P(\cdot|J_{(1, \infty)}))=\frac 56$, we have
\[V_n\geq \int_{J_{(1, \infty)}}(x-\frac 56)^2 dP=\frac{25}{3672}=0.00680828>V_n,\]
which leads to a contradiction. So, we can assume that $\te{card}(\ga_n\ii [S_2(0), 1])\geq 2$. Next, suppose that $\te{card}(\ga_n\ii J_1)=1$. Then, as $E(P(\cdot|J_1))=\frac 16$, we have
\[V_n\geq \int_{J_1}(x-\frac 16)^2 dP=\frac{1}{216}=0.00462963>V_n,\]
which leads to another contradiction. Thus, the proof of the proposition is complete.
\end{proof}

\begin{remark}\label{rem1}
From Proposition~\ref{prop23}, it follows that if $\ga_n$ is an optimal set of four-means, then  $\te{card}(\ga_n\ii J_1)= 2$ and $\te{card}(\ga_n\ii [J_2(0), 1])= 2$.
\end{remark}

\begin{prop} \label{prop214} Let $\ga_n$ be an optimal set of $n$-means for $P$ such that $\te{card}(\ga_n\ii [J_{k+1}(0), 1])\geq 2$ for some $k\in \D N$ and $n\in \D N$. Then,

$(i)$  $\ga_n\ii J_{k+1}\neq \es$ and $\ga_n\ii [J_{k+2}(0), 1]\neq \es$;

$(ii)$ $\ga_n$ does not contain any point from the open interval $(J_{k+1}(1),  J_{k+2}(0))$;

$(iii)$ the Voronoi region of any point in $\ga_n\ii J_{k+1}$ does not contain any point from $[J_{k+2}(0), 1]$ and the Voronoi region of any point in $\ga_n\ii [J_{k+2}(0), 1]$ does not contain any point from $J_{k+1}$.
\end{prop}

\begin{proof} To prove the proposition it is enough to prove it for $k=1$, and then inductively the proposition will follow for all $k\geq 2$. Fix $k=1$.
Suppose that $\te{card}(\ga_n \ii [J_2(0), 1])\geq 2$. By Lemma~\ref{lemma22}, it is clear that the proposition is true for $n=3$. We now prove  that the proposition is true for $n\geq 4$.
Let $\ga_n:=\set{0<a_1<a_2<\cdots<a_n<1}$ be an optimal set of $n$-means for any $n\geq 4$. Let $V(P, \ga_n\ii [J_2(0),1])$ be the quantization error contributed by the set $\ga_n\ii [J_2(0),1]$ in the region $[J_2(0),1]$. Let $\gb$ be a set such that $\gb:=\set{\frac 1{12}, \frac 14, \frac {13}{18}, \frac{17}{18}}$. The distortion error due to the set $\gb\ii [J_2(0),1]:=\set{\frac {13}{18}, \frac{17}{18}}$ is given by
\[\int_{[J_2(0),1]}\min_{a \in \gb\ii [J_2(0),1]}(x-a)^2 dP=\int_{J_2}(x-\frac {13}{18})^2dP+\int_{J_{(2, \infty)}}(x-\frac{17}{18})^2dP=\frac{7}{11016}=0.000635439,\]
and so,  $V(P, \ga_n\ii [J_2(0),1])\leq 0.000635439$.
Suppose that $\ga_n$ does not contain any point from $J_2$. Since by Proposition~\ref{prop212}, the Voronoi region of any point in $\ga_n\ii J_1$ does not contain any point from $[J_2(0),1]$, we have
 \[V(P, \ga_n\ii [J_2(0),1])\geq \int_{J_2}(x-\frac{7}{9})^2 dP=\frac{1}{972}=0.00102881>V(P, \ga_n\ii [J_2(0),1]),\]
 which leads to a contradiction. So, we can assume that $\ga_n\ii J_2\neq \es$. Suppose that $\ga_n\ii [J_{3}(0), 1]= \es$.
Then, $a_n<J_3(0)=\frac 89$, and so,
\[V(P, \ga_n\ii [J_2(0),1])\geq  \sum_{j=3}^\infty \int_{J_j}(x-J_3(0))^2 dP=\frac{19}{16524}=0.00114984>V(P, \ga_n\ii [J_2(0),1]),\]
which gives another contradiction. Therefore, $\ga_n\ii [J_{3}(0), 1]\neq \es$, i.e., $(i)$ is proved.

To prove $(ii)$ we proceed as follows: If $\te{card}(\ga_n \ii [J_2(0), 1])= 2$, then as Lemma~\ref{lemma21}, it can be proved that $\ga_n \ii [J_2(0), 1]=\set{E(P(\cdot|J_2)), E(P(\cdot|J_{(2, \infty)}))}=\set{\frac{13}{18}, \frac{17}{18}}$. Since $\frac{13}{18}\in J_2$ and $J_3(1)=\frac{8}{9}<\frac {17}{18}$, in this case we see that $\ga_n\ii (J_2(1), J_3(0))=\es$. If $\te{card}(\ga_n \ii [J_2(0), 1])=3$, then as Lemma~\ref{lemma22}, it can be proved that
\[\ga_n \ii [J_2(0), 1]=\set{E(P(\cdot|J_2)),  E(P(\cdot|J_3)), E(P(\cdot|J_{(3, \infty)}))}=\set{\frac{13}{18}, \frac{49}{54}, \frac{53}{54}}\] implying the fact that $\ga_n\ii (J_2(1), J_3(0))=\es$. We now assume that $\te{card}(\ga_n \ii [J_2(0), 1])=4$, then as mentioned in Remark~\ref{rem1}, in this case, we can also prove that $\te{card}(\ga_n \ii J_2)=2$ and $\te{card}(\ga_n \ii [J_3(0), 1])=2$, in fact, we have $\te{card}(\ga_n \ii [J_2(0), 1])=\set{\frac{25}{36}, \frac 34, \frac{49}{54}, \frac{53}{54}}$ implying $\ga_n\ii (J_2(1), J_3(0))=\es$, and the corresponding quantization error, by Proposition~\ref{prop1}, is given by
\[V(P, \ga_2(P(\cdot|J_2)), J_2)+V(P, \ga_1(P(\cdot|J_3)), J_3)+V(P, \set{E(P(\cdot|J_{(3, \infty)}))}, J_{(3, \infty)})=\frac{79}{793152}.\]
Next, assume that $\te{card}(\ga_n \ii [J_2(0), 1])\geq 4$. Then, we must have $V(P, \ga_n\ii [J_2(0),1])\leq \frac{79}{793152}=0.0000996026$.
Let $j:=\max\set{i : a_i \leq J_2(1) \te{ for } 1\leq i\leq n}$ implying $a_j\leq \frac{7}{9}=J_2(1)$. Suppose that $\frac 79<a_{j+1}<\frac 89$. The following cases can arise:

Case~1.  $\frac 79<a_{j+1}<\frac 56$.

Then, $\frac 12(a_{j+1}+a_{j+2})>\frac 89$ implying $a_{j+2}>\frac {16}{9}-a_{j+1}\geq \frac {16}{9}-\frac {5}{6}=\frac{17}{18}>J_3(1)$, and so,
\[V(P, \ga_n\ii [J_2(0),1]) \geq \int_{J_3}(x-\frac{17}{18})^2 dP=\frac{13}{69984}=0.000185757>V(P, \ga_n\ii [J_2(0),1]),\]
which is contradiction.

Case~2. $\frac{5}{6} \leq a_{j+1}< \frac 89.$

Then, $\frac 12(a_{j}+a_{j+1})<\frac 79$ implying $a_{j}<\frac {14}9-a_{j+1}\leq \frac {14}9-\frac{5}{6}=\frac{13}{18}$, and so,
\[V(P, \ga_n\ii [J_2(0),1]) \geq \int_{[\frac {13}{18}, \frac 79]}(x-\frac{13}{18})^2 dP=\frac{1}{7776}=0.000128601>V(P, \ga_n\ii [J_2(0),1]),\]
which gives a contradiction.

Thus, $\ga_n\ii (J_2(1), J_3(0))=\es$, which completes the proof of $(ii)$.  The proof of $(iii)$ is similar to the proof of $(iii)$ in Proposition~\ref{prop212}.
Hence, the proposition is yielded.
\end{proof}

\begin{prop}\label{prop215}
Let $\ga_n$ be an optimal set of $n$-means for $n\geq 2$. Then, there exists a positive integer $k:=k(n)$ such that  $\ga_n\ii J_j \neq \es$ for all $1\leq j\leq k$, and $\te{card}(\ga_n\ii [J_{k+1}(0), 1])=1$. Write $\ga_{n,j}:=\ga_n\ii J_j$ and $n_j:=\te{card}(\ga_{n,j})$. Then, $\ga_{n, j}=\ga_{n_j}(P(\cdot|J_j))$ and $n=\sum_{j=1}^k n_j+1$, with
\[V_n=
\mathop\sum\limits_{j=1}^k V(P, \ga_{n_j}(P(\cdot|J_j)), J_j)+V(P, \ga_1(P(\cdot|J_{(k, \infty)})), J_{(k, \infty)})=\sum_{j=1}^k\frac 1{n_j^2}\frac 1{12} \frac{1}{18^j}+\frac{25}{204}\frac 1{18^k}.\]
\end{prop}

\begin{proof} Proposition~\ref{prop212} says that if $\ga_n$ is an optimal set of $n$-means for $n\geq 2$, then $\ga_n\ii J_1\neq \es$, $\ga_n\ii [J_2(0), 1]\neq \es$, and $\ga_n$ does not contain any point from the open interval $(J_1(1), J_2(0))$. Proposition~\ref{prop214} says that if $\te{card}(\ga_n\ii [J_{k+1}(0), 1])\geq 2$ for some $k\in \D N$, then $\ga_n\ii J_{k+1}\neq \es$ and $\ga_n\ii [J_{k+2}(0), 1]\neq \es$. Moreover, $\ga_n$ does not take any point from the open interval $(J_{k+1}(1), J_{k+2}(0))$. Thus, by Induction Principle, we can say that if $\ga_n$ is an optimal set of $n$-means for $n\geq 2$, then there exists a positive integer $k$ such that $\ga_n\ii J_j\neq \es$ for all $1\leq j\leq k$ and $\te{card}(\ga_n\ii [J_{k+1}(0), 1])=1$.

For a given $n\geq 2$, write $\ga_{n,j}:=\ga_n\ii J_j$ and $n_j:=\te{card}(\ga_{n,j})$. Since the Voronoi region of any point in $\ga_{n,j}$ does not contain any point from $J_1, J_2, \cdots, J_{j-1}$, and $J_{(j, \infty)}$, we must have $\ga_{n,j}=\ga_{n_j}(P(\cdot|J_j))$. Again, $\ga_{n,j}$ are disjoint for $1\leq j\leq k$ and $\ga_n$ does not contain any point from the open intervals $(J_{\ell}(1), J_{\ell+1}(0))$ for $1\leq \ell\leq k$. This implies the fact that
$\ga_n=\mathop{\uu}\limits_{j=1}^k\ga_{n,j}\uu \set{\ga_1(P(\cdot|J_{(k, \infty)}))}$ and $n=n_1+n_2+\cdots +n_k+1$, and so,
\begin{align*} &V_n=\int \min_{a \in \ga_{n}} (x-a)^2 dP=\sum_{j=1}^{k} \int_{J_j}\min_{a \in \ga_{n,j}} (x-a)^2 dP+\int_{J_{(k, \infty)}}(x-\ga_1(P(\cdot|J_{(k, \infty)})))^2 dP\\
&=\sum_{j=1}^{k} V(P, \ga_{n_j}(P(\cdot|J_j)), J_j)+V(P, \ga_1(P(\cdot|J_{(k, \infty)})), J_{(k, \infty)})=\sum_{j=1}^k\frac 1{n_j^2}\frac 1{12} \frac{1}{18^j}+\frac{25}{204}\frac 1{18^k}.
\end{align*}
Thus, the proof of the proposition is complete.
\end{proof}

\begin{defi} \label{defi1} Let $n_j$ for $1\leq j\leq k$ be the positive integers as defined in Proposition~\ref{prop215}. Then, we call the sequence $\set{n_1, n_2, \cdots, n_k, 1}$ a \tit{canonical sequence of order $n$} or just a \tit{canonical sequence}. Notice that once a canonical sequence of order $n$ is known the corresponding  optimal set of $n$-means can easily be determined and vice versa. Let $\set{n_1, n_2, \cdots, n_k, 1}$ be a canonical sequence and $m\in \D N$ with $1\leq m\leq k$. Then, the sequence $\set{n_m, n_{m+1}, \cdots, n_k, 1}$ is called a \tit{subblock} of the canonical sequence $\set{n_1, n_2, \cdots, n_k, 1}$.
\end{defi}

The canonical sequence has the following property.

\begin{lemma}\label{lemma65}  Let $\set{n_1, n_2, \cdots, n_k, 1}$ be a canonical sequence for $k\geq 2$. Then, $n_1> n_2> n_3 > \cdots >n_{k-1}\geq n_k=1$.
\end{lemma}

\begin{proof} Let $\ga_n$ be an optimal set of $n$-means, and $\set{n_1, n_2, \cdots, n_k, 1}$ be the canonical sequence associated with $\ga_n$. Take any $1\leq i<k$. Let $n_i+n_{i+1}=m$. Notice that $m$ is constant if $i$ remains fixed. The distortion error in the intervals $J_i$ and $J_{i+1}$ is given by
\begin{align} \label{eq67}
&V(P, \ga_{n_i}(P(\cdot|J_i)), J_i)+V(P, \ga_{n_{i+1}}(P(\cdot|J_{i+1})), J_{i+1})\notag\\
&=\frac 1{12} \Big(\frac {1}{n_i^2}\frac 1{18^i}+\frac {1}{(m-n_i)^2}\frac 1{18^{i+1}}\Big) \\
&=\frac 1{12}\frac 1{18^{i+1}}\Big(\frac{18 m^2-36 m n_i+19 n_1^2}{n_i^2 (m-n_i)^2}\Big),\notag
\end{align}
which is minimum if $n_i\approx  \frac{1}{19} (18 \, m-3\sqrt[3]{12} \, m+\sqrt[3]{18}\, m)$, where for any positive real number $x$, by $n_i\approx x$ it is meant that $n_i$ is the positive integer nearest to $x$. Then, notice that $m=2$ implies $n_i=n_{i+1}=1$, and if $m\geq 3$ then $n_i>\frac m 2$ yielding $n_i>n_{i+1}$.
By Proposition~\ref{prop215}, it follows that $n_k=1$, and thus, the lemma is yielded.
\end{proof}

\begin{remark} \label{remark71}
From Table~\ref{tab1}, we see that $\set{6, 3, 1, 1}$ is a canonical sequence, where $n_1=6$, $n_2=3$ and $n_3=1$. Take $m=n_1+n_2=9$, then $\frac{1}{19} (18 \, m-3\sqrt[3]{12} \, m+\sqrt[3]{18}\, m)=6.51432\approx 7\neq n_1$. Thus, we see that the canonical sequence $\set{6, 3, 1, 1}$ violates the statement $n_i\approx  \frac{1}{19} (18 \, m-3\sqrt[3]{12} \, m+\sqrt[3]{18}\, m)$ as mentioned in the proof of Lemma~\ref{lemma65}. But, such a canonical sequence does not occur frequently, and it does not violate the statement of Lemma~\ref{lemma65}. Putting $i=1$ and $m=9$ in the expression \eqref{eq67}, we see that it is minimum if $n_1=6$, which is the value that occurs in the canonical sequence $\set{6, 3, 1, 1}$. Hence, if $m$ and $i$ are known, using the expression \eqref{eq67} one can exactly determine $n_i$.
\end{remark}

We now give the following example.

\begin{example} By Lemma~\ref{lemma22}, for $n=3$, we have $\ga_3=\set{\frac 1 6, \frac{13}{18}, \frac{17}{18}}$ implying $\ga_{3, 1}=\set{\frac 16}$ and $\ga_{3, 2}=\set{\frac{13}{18}}$, and $\ga_1(P(\cdot|J_{(2, \infty)}))=\set{\frac{17}{18}}$. Here the canonical sequence is $\set{1, 1, 1}$. By Proposition~\ref{prop215},
\begin{align*} V_3&=V(P, \ga_{3, 1}, J_1)+V(P, \ga_{3, 2}, J_2)+V(P, \ga_1(P(\cdot|J_{(2, \infty)})), J_{(2, \infty)}),
\end{align*}
and so,  by Proposition~\ref{prop1}, $V_3=\frac 1{1^2} \frac 1{12}\frac 1{18}+\frac 1{1^2} \frac 1{12}\frac 1{18^2}+\frac {25}{204}\frac 1{18^2}=\frac{29}{5508}$, which is the quantization error for three-means obtained in Lemma~\ref{lemma22}.
\end{example}

The following lemma gives some more properties of canonical sequences.

\begin{lemma} \label{lemma53a}
Let $n\in \D N$ and $n\geq 2$. Then, $(i)$ a canonical sequence of order $n$ is unique, and $(ii)$ each subblock of a canonical sequence is also a canonical sequence.
\end{lemma}

\begin{proof}
Let $\set{n_1, n_2, \cdots, n_k, 1}$ be a canonical sequence of order $n$. For the sake of contradiction assume that $\set{n_1', n_2', \cdots, n_k', 1}$ is another canonical sequence of order $n$. Then, we must have indices $i_1, i_2, i_3$ such that $n_{i_2}\neq n_{i_2}'$, but $n_{i_1}+n_{i_2}>n_{i_1}'+n_{i_2}'$ and $n_{i_2}+n_{i_3}<n_{i_2}'+n_{i_3}'$. Putting $m=n_{i_1}+n_{i_2}$ in the expression similar to \eqref{eq67}, we can uniquely determine $n_{i_1}$ and $n_{i_2}$. Similarly, putting $m=n_{i_1}'+n_{i_2}'$, we can uniquely determine $n_{i_1}'$ and $n_{i_2}'$. Since $n_{i_1}+n_{i_2}>n_{i_1}'+n_{i_2}'$, we will have $n_{i_1}\geq n_{i_1}'$ and $n_{i_2}\geq n_{i_2}'$. Similarly, $n_{i_2}+n_{i_3}<n_{i_2}'+n_{i_3}'$ implies $n_{i_2}\leq n_{i_2}'$ and $n_{i_3}\leq n_{i_3}'$. Thus, we see that $n_{i_2}\geq n_{i_2}'$ and $n_{i_2}\leq n_{i_2}'$ yield a contradiction to our assumption that $n_{i_2}\neq  n_{i_2}'$. Therefore, we can assume that the canonical sequence of order $n$ is unique, which completes the proof of $(i)$. To prove $(ii)$, we proceed as follows: Let  $\set{n_1, n_2, \cdots, n_k, 1}$ be the canonical sequence of order $n$. It is enough to show that $\set{n_2, n_3, \cdots, n_k, 1}$ is the canonical sequence of order $n-n_1$. For the sake of contradiction, assume that $\set{n_2', n_3', \cdots, n_k', 1}$ is the canonical sequence of order $n-n_1$. Since a canonical sequence of a given order is unique, if we calculate the quantization error, we must have
\[\sum_{j=2}^k \frac{1}{n_j^2}\frac 1{12}\frac 1{18^j}>\sum_{j=2}^k \frac{1}{n_j'^2}\frac 1{12}\frac 1{18^j}
\te{ implying } \sum_{j=1}^k \frac{1}{n_j^2}\frac 1{12}\frac 1{18^j}>\frac 1{n_1^2}\frac 1{12}\frac 1{18}+\sum_{j=2}^k \frac{1}{n_j'^2}\frac 1{12}\frac 1{18^j},\]
which contradicts the fact that $\set{n_1, n_2, \cdots, n_k, 1}$ is the canonical sequence of order $n$. Hence, every subblock of a canonical sequence is also a canonical sequence.
\end{proof}

\begin{lemma} \label{lemma53b}
 Let $\set{n_1, n_2, \cdots, n_k, 1}$ be the canonical sequence of order $n$ for $n\in \D N$ and $n\geq 2$. Then, the canonical sequence of order $(n+1)$ will be either $\set{n_1, n_2, \cdots, n_{i-2}, n_{i-1}, n_i+1, n_{i+1}, \cdots,  n_{k-1}, n_k, 1}$ for some $1\leq i\leq k-1$, or $\set{n_1, n_2, \cdots, n_k, 1, 1}$.
\end{lemma}
\begin{proof}
We prove the lemma by induction. By Lemma~\ref{lemma21} and Lemma~\ref{lemma22}, the canonical sequences of order two and three are $\set{1, 1}$ and $\set{1, 1, 1}$, respectively. Again, by Remark~\ref{rem1}, it can be seen that the canonical sequence of order four is $\set{2, 1, 1}$. Thus, we see that the lemma is true for $n=2$ and $n=3$. Let $N\geq 4$ be a positive integer such that the lemma is true for all positive integers $n$, where $2\leq n\leq N-1$. We will show that the lemma is also true for $n=N$. Let $\set{n_1, n_2, \cdots, n_k, 1}$ be the canonical sequence of order $N$ implying that the optimal set $\ga_N$ contains $n_1+n_2+\cdots+n_k$ elements from $J_1\uu J_2\uu \cdots \uu J_k$ and one element from $J_{(k, \infty)}$.
Then, the optimal set $\ga_{N+1}$ contains exactly one or two elements from $J_{(k, \infty)}$. Assume that $\ga_{N+1}$ contains two elements from $J_{(k, \infty)}$. Since $\set{1, 1}$ is the only subblock of order two, the canonical sequence of order $(N+1)$ is $\set{m_1, m_2, \cdots, m_k, 1, 1}$. Again, as $m_1+m_2+\cdots+m_k=n_1+n_2+\cdots+n_k=N-1$ and the canonical sequence of order $N$ is unique, we must have $m_1=n_1, \, m_2=n_2, \, \cdots, \, m_k=n_k$. Thus, in this case the lemma is true. Now, assume that $\ga_{N+1}$ contains only one element from $J_{(k, \infty)}$.
In this case the canonical sequence of order $(N+1)$ is $\set{m_1, m_2, \cdots, m_k, 1}$. We need to show that $m_j=n_j+1$ for exactly one $1\leq j\leq k$, and $m_j=n_j$ for all other $1\leq j\leq k$. First, assume that $m_1=n_1$. Then, both $\set {m_2, m_3, \cdots, m_k, 1}$ and $\set{n_2, n_3, \cdots, n_k, 1}$ are canonical sequences of order $N+1-m_1$ and $N-n_1$ respectively. Since $(N+1-m_1)-(N-n_1)=1$, and we assumed that the lemma is true for all positive integers $n\leq N-1$, we have $m_j=n_j+1$ for exactly one $2\leq j\leq k$, and $m_j=n_j$ for all other $2\leq j\leq k$, which combined with $m_1=n_1$ yields that the lemma is true for $n=N$. If $m_1=n_1+1$, then as both $\set {m_2, m_3, \cdots, m_k, 1}$ and $\set{n_2, n_3, \cdots, n_k, 1}$ are canonical sequences of the same order, we have $m_2=n_2, m_3=n_3, \cdots, m_k=n_k$, which combined with $m_1=n_1+1$ yields that the lemma is true for $n=N$. We now show that $m_1$ can not be any integer other than $n_1$ or $n_1+1$. For the sake of contradiction, assume that $m_1=n_1+k$ for some $k\geq 2$.
Then, $\set{m_2, \cdots, m_{k}, 1}$ is the canonical sequence of order $N+1-m_1=N+1-(n_1+k)=N-n_1-(k-1)$, and $\set{n_2, n_3, \cdots, n_k, 1}$ is the canonical sequence of order $N-n_1$. Since we assumed that the lemma is true for all positive integers $n\leq N-1$, we must have $n_j> m_j$ for at least one $2\leq j\leq k$. Without any loss of generality, assume that $n_2>m_2$ and then $n_2=m_2+\ell$ for some $1\leq \ell\leq  (k-1)$, and so,  $m_1+m_2=n_1+n_2+(k-\ell)>n_1+n_2$, which by an expression similar to \eqref{eq67} implies that $m_1\geq n_1$ and $m_2\geq n_2$ yielding a contradiction. Similarly, we can show that if $m_1=n_1-k$ for any $k\in \D N$, a contradiction arises. Thus, the lemma is true for $n=N$ if it is true for all positive integers $n\leq N-1$. Hence, by the principle of Mathematical Induction the proof of the lemma is complete.
\end{proof}

\begin{table}
%\caption {Table Title} \label{tab:title}
\begin{center}
\begin{tabular}{ |c|c||c|c|| c|c|}
 \hline
$n$ & $\te{canonical sequence} $ & $n$ & $\te{canonical sequence}$  & $n$ & $\te{canonical sequence} $  \\
\hline
2 & \set{1, 1}  & 21 & \set{12, 5, 2, 1, 1} & 40 & \set{24, 9, 4, 1, 1, 1}\\3 & \set{1, 1, 1} & 22 & \set{13, 5, 2, 1, 1} & 41 & \set{25, 9, 4, 1, 1, 1} \\ 4 & \set{2, 1, 1}  & 23 &\set{14, 5, 2, 1, 1}& 42 & \set{25, 10, 4, 1, 1, 1}\\5 & \set{3, 1, 1} &  24 &\set{14, 6, 2, 1, 1} & 43 & \set{25, 10, 4, 2, 1, 1} \\6 & \set{3, 1, 1, 1} & 25 &\set{15, 6, 2, 1, 1} & 44 & \set{26, 10, 4, 2, 1, 1}\\ 7 & \set{4, 1, 1, 1} &  26 &\set{16, 6, 2, 1, 1} & 45 & \set{27, 10, 4, 2, 1, 1} \\8 & \set{4, 2, 1, 1} &  27 & \set{17, 6, 2, 1, 1} & 46 & \set{27, 11, 4, 2, 1, 1}\\9 & \set{5, 2, 1, 1} &   28 & \set{17, 6, 3, 1, 1} & 47 & \set{28, 11, 4, 2, 1, 1}\\10 & \set{6, 2, 1, 1} &  29 &\set{17, 7, 3, 1, 1}   & 48 & \set{29, 11, 4, 2, 1, 1}\\11 & \set{6, 3, 1, 1} &30 & \set{18, 7, 3, 1, 1}& 49 & \set{30, 11, 4, 2, 1, 1} \\12 & \set{7, 3, 1, 1} & 31 & \set{19, 7, 3, 1, 1} & 50 & \set{30, 12, 4, 2, 1, 1} \\13 & \set{8, 3, 1, 1} & 32 & \set{20, 7, 3, 1, 1} & 51 & \set{31, 12, 4, 2, 1, 1} \\14 & \set{8, 3, 1, 1, 1} &  33 & \set{20, 8, 3, 1, 1} & 52 & \set{31, 12, 5, 2, 1, 1} \\ 15 &  \set{9, 3, 1, 1, 1} & 34 & \set{21, 8, 3, 1, 1}  & 53 & \set{32, 12, 5, 2, 1, 1}\\ 16 & \set{9, 4, 1, 1, 1} & 35 & \set{21, 8, 3, 1, 1, 1} & 54 & \set{33, 12, 5, 2, 1, 1}\\ 17 & \set {10, 4, 1, 1, 1} & 36 & \set{22, 8, 3, 1, 1, 1} & 55 & \set{33, 13, 5, 2, 1, 1}\\ 18 & \set{ 10, 4, 2, 1, 1} & 37 & \set{22, 9, 3, 1, 1, 1} & 56 & \set{34, 13, 5, 2, 1, 1} \\ 19 &\set{11, 4, 2, 1, 1} & 38 & \set{23, 9, 3, 1, 1, 1}& 57 & \set{35, 13, 5, 2, 1, 1}\\  20 &\set{ 12, 4, 2, 1, 1} & 39 & \set{24, 9, 3, 1, 1, 1} & 58 & \set{35, 14, 5, 2, 1, 1}\\
 \hline
\end{tabular}
 \end{center}
 \
\caption{List of canonical sequences for the optimal sets $\ga_n$ in the range $2\leq n\leq 58$.}
    \label{tab1}
\end{table}
We are now ready to state and prove the following theorem which gives the optimal set of $(n+1)$-means whenever the optimal set of $n$-means is known.

\begin{theorem}\label{Th1}
Let $\set{n_1, n_2, \cdots, n_k, 1}$ be the canonical sequence for an optimal set of $n$-means for some $n\in \D N$. Construct the sequence $\set{A(i)}_{i=1}^k$ such that
\[A(i)=\set{n_1, n_2, \cdots,n_{i-1}, n_i+1, n_{i+1}, \cdots, n_k}\]
for $1\leq i\leq k$. For $1\leq i\leq k$, set
\begin{align*} V(A(i))& :=\sum_{\substack{j=1\\ j\neq i}}^k \frac 1{n_j^2} \frac1 { 12} \frac 1{18^j}+\frac 1{(n_i+1)^2} \frac1 { 12}  \frac 1{18^j}+\frac{25}{204}\frac 1{18^k}, \te{ and } \\
V(\infty)&:=\sum_{j=1}^{k} \frac 1{n_j^2} \frac1 { 12}  \frac 1{18^j} +\frac 1{1^2}\frac 1{12}\frac 1{18^{k+1}}+\frac{25}{204}\frac 1{18^{k+1}}.
\end{align*}
Write $V_{\min}:=\min\set{\min\set{V(A(j)) : 1\leq j\leq k}, V(\infty)}$.
If $V_{\min}:=V(A(m))$ for some $1\leq m\leq k$, then the sequence $\set{n_1, n_2, \cdots, n_{m-1}, n_m+1, n_{m+1}, \cdots, n_k, 1}$ is the canonical sequence which gives an optimal set of $(n+1)$-means. If $V_{\min}=V(\infty)$, then $\set{n_1, n_2, \cdots, n_k, 1, 1}$ is the canonical sequence which gives an optimal set of $(n+1)$-means.
\end{theorem}

\begin{proof} By Lemma~\ref{lemma21}, we see that $\set{1, 1}$ is the canonical sequence for an optimal set of two-means and $\set{1, 1, 1}$ is the canonical sequence for an optimal set of three-means. In fact, for the canonical sequence $\set{1, 1}$, we have
$V(A(1))=\frac 1{2^2}\frac 1{12}\frac 1{18}+\frac{25}{204}\frac 1{18}=\frac{13}{1632}$ and $V(\infty)=\frac 1{1^2}\frac 1{12}\frac 1{18}+\frac 1{1^2}\frac 1{12}\frac 1{18^2}+\frac{25}{204}\frac 1{18^2}=\frac{29}{5508}$ implying $V(\infty)<V(A(1))$. Thus, we see that the theorem is true if $k=1$. Let us now assume that $\set{n_1, n_2, \cdots, n_k, 1}$ is the canonical sequence for an optimal set of $n$-means for $n\in \D N$. Then, using the hypothesis of the theorem, and Lemma~\ref{lemma53b}, the proof of the theorem is complete.
\end{proof}

%%\begin{note} \label{note2}  From Lemma~\ref{lemma21} and Lemma~\ref{lemma22}, we know the canonical sequence for optimal sets of two- and three-means.
%%Notice that once a canonical sequence is known the corresponding optimal set of $n$-means is known and vice versa. Let $\set{n_1, n_2, \cdots, n_k, 1}$ be the canonical sequence associated with an optimal set of $n$-means. Then, to determine an optimal set of $(n+1)$-means we proceed as follows:
%%
%%$(i)$ Construct the sequence $\set{A(i)}_{i=1}^k$ such that $A(i)=\set{n_1, n_2, \cdots,n_{i-1}, n_i+1, n_{i+1}, \cdots, n_k}$.
%%
%%$(i)$ Calculate the errors $V(A(i))$ and $V(\infty)$, where for $1\leq i\leq k$,
%%\begin{align*} V(A(i))& :=\sum_{\substack{j=1\\ j\neq i}}^k \frac 1{n_j^2} \frac 12 \frac 1{18^j}+\frac 1{(n_i+1)^2} \frac 12 \frac 1{18^j}+\frac{25}{204}\frac 1{18^k}, \te{ and } V(\infty):=\sum_{j=1}^{k} \frac 1{n_j^2} \frac 12 \frac 1{18^j}+\frac{25}{204}\frac 1{18^{k+1}}.
%%\end{align*}
%%
%%$(iii)$ Find $V_{\min}:=\min\set{\min\set{V(A(j)) : 1\leq j\leq k}, V(\infty)}$.
%%
%%$(iv)$ If $V_{\min}:=V(A(m))$ for some $1\leq m\leq k$, then the sequence $\set{n_1, n_2, \cdots, n_{m-1}, n_m+1, n_{m+1}, \cdots, n_k, 1}$ is a canonical sequence which gives an optimal set of $(n+1)$-means. If $V_{\min}=V(\infty)$, then $\set{n_1, n_2, \cdots, n_k, 1, 1}$ is a canonical sequence which gives an optimal set of $(n+1)$-means.
%%\end{note}

\begin{remark}
Using Theorem~\ref{Th1}, we obtain Table~\ref{tab1} which gives a list of canonical sequences of order $n$ for $2\leq n\leq 58$. Notice that for any positive integer $n\in \D N$, $n\geq 2$, to obtain the canonical sequence of order $(n+1)$ one needs to know the canonical sequence of order $n$. A closed formula to obtain the canonical sequence of any order $n\in \D N$ is still not known. On the other hand, in the following section, we show that for a piecewise uniform distribution with finitely many pieces we can easily determine the optimal sets of $n$-means and the $n$th quantization errors for all $n\in \D N$, see Note~\ref{note12}.
\end{remark}

\section{Optimal Quantization for Uniform Distribution with Finitely Many Pieces} \label{sec5}

Most of the notations and basic definitions used in this section are same as they are described in Section~\ref{sec3}. Write $J_1=[0, \frac 13]$, $J_2=[\frac 23, \frac 79]$ and $J_3=[\frac 89, 1]$. Let $P$ be a piecewise uniform distribution on the real line with probability density function (pdf) $f(x)$ given by
\[f(x)=\left\{\begin{array}{lll}
\frac 32 & \te{ if } x \in J_1, \\

\frac 9 4 & \te{ if } x \in J_2\uu J_3, \\

 0  & \te{ otherwise}.
\end{array}\right.
\]

\begin{lemma}\label{lemma111}
Let $E(P)$ and $V(P)$ represent the expected value and the variance of a random variable $X$ with distribution $P$. Then, $E(P)=\frac 12$ and $V(P)=\frac{119}{972}$.
\end{lemma}

\begin{proof} We have
\begin{align*}
E(P) & =\int x dP=\int_{J_1} \frac{3 x}{2} \, dx+\int_{J_2} \frac{9 x}{4} \, dx+\int_{J_3} \frac{9 x}{4} \, dx=\frac 12, \te{ and } \\
V(P)& =\int(x-\frac 12)^2 dP=\int_{J_1} \frac{3}{2}(x-\frac 12)^2 dx+\int_{J_2} \frac{9}{4}(x-\frac 12)^2 \, dx+\int_{J_3} \frac{9 }{4}(x-\frac 12)^2 \, dx=\frac{119}{972},
\end{align*}
and thus the lemma is yielded.
\end{proof}

\begin{lemma} \label{lemma222a}
For $k=1, 2, 3$, let $E(P(\cdot|J_k))$ denote the expectations of the random variable $X$ with distributions $P(\cdot|J_k)$. Then,
\[E(P(\cdot|J_1))=\frac 16,  \, E(P(\cdot|J_2))=\frac{13}{18} \te{ and } E(P(\cdot|J_3))=\frac{17}{18}.\]
\end{lemma}

\begin{proof} By the definition of the conditional expectation, we have
\[E(P(\cdot|J_1))=\int_{J_1} x dP(\cdot|J_1)=\frac 1{P(J_1)} \int_{J_1}x dP =2 \int_{J_1} \frac 3 2 x dx=\frac 16, \te{ and similarly } \]
 we can obtain  $E(P(\cdot|J_2))=\frac{13}{18} \te{ and } E(P(\cdot|J_3))=\frac{17}{18}$.
Hence, the lemma is yielded.
\end{proof}

The following proposition is similar to Proposition~\ref{prop1}.
\begin{prop} \label{prop111}
Let $n\in \D N$. Then, the set $\set{\frac{2i-1}{2n}\frac 1{3} : 1\leq i\leq n}$ is a unique optimal set of $n$-means for $P(\cdot|J_1)$, i.e., $\ga_n(P(\cdot|J_1))=\set{\frac{2i-1}{2n}\frac 1{3} : 1\leq i\leq n}$. Similarly, $\ga_n(P(\cdot|J_2))=\set{\frac 23+\frac{2i-1}{2n}\frac 1{9} : 1\leq i\leq n}$ and $\ga_n(P(\cdot|J_3))=\set{\frac 89+\frac{2i-1}{2n}\frac 1{9} : 1\leq i\leq n}$. Moreover,
\[V(P, \ga_n(P(\cdot|J_1)), J_1)=\frac{1}{216 n^2} \te{ and } V(P, \ga_n(P(\cdot|J_2)), J_2)=V(P, \ga_n(P(\cdot|J_3)), J_3)=\frac{1}{3888 n^2}.\]
\end{prop}
%%\begin{proof}
%%Since $P(\cdot|J_1)$ is uniformly distributed on $J_1$, the boundaries of the Voronoi regions of an optimal set of $n$-means will divide the interval $[0, \frac 13]$ into $n$ equal subintervals, i.e., the boundaries of the Voronoi regions are given by the points
%%\[0, \ \frac {1}{n}\frac 1{3}, \ \frac {2}{n}\frac 1{3},\ \cdots, \frac {n-1}{n}\frac 1{3},\  \frac 1{3}.\]
%%This implies that an optimal set of $n$-means for $P(\cdot|J_1)$ is unique, and it consists of the midpoints of the boundaries of the Voronoi regions, i.e., the optimal set of $n$-means for $P(\cdot|J_1)$ is given by $\set{\frac{2i-1}{2n}\frac 1{3} : 1\leq i\leq n}$ for any $n\geq 1$. Then, the $n$th quantization error for $P$ due to the set $\ga_n(P(\cdot|J_1))$ on $J_1$ is given by
%%\begin{align*}
%%&V(P, \ga_n(P(\cdot|J_1)), J_1)=n \times (\te{the quantization error in each Voronoi region})\\
%%&=n \Big(\int_{[0, \frac{1}{n}\frac 1{3}]} \frac 32 \Big(x-\frac{1}{2n}\frac 1{3}\Big)^2  dx\Big),
%%\end{align*}
%%which after simplification implies $V(P, \ga_n(P(\cdot|J_1)), J_1)= \frac{1}{216 n^2}$. Similarly, we can show that $\ga_n(P(\cdot|J_2))=\set{\frac 23+\frac{2i-1}{2n}\frac 1{9} : 1\leq i\leq n}$ and $\ga_n(P(\cdot|J_3))=\set{\frac 89+\frac{2i-1}{2n}\frac 1{9} : 1\leq i\leq n}$, and
%%\[V(P, \ga_n(P(\cdot|J_2)), J_2)=V(P, \ga_n(P(\cdot|J_3)), J_3)=\frac{1}{3888 n^2}.\]
%%Thus, the proof of the proposition is complete.
%%\end{proof}

The following two lemmas are similar to Lemma~\ref{lemma21} and Lemma~\ref{lemma22}.
\begin{lemma}  \label{lemma211}
Let $\ga:=\set{a_1, a_2}$ be an optimal set of two-means such that $a_1<a_2$. Then, $a_1= \frac 16$ and $a_2=\frac 56$, and the corresponding quantization error is $V_2=\frac{11}{972}$.
\end{lemma}

\begin{lemma}  \label{lemma222b}
Let $\ga:=\set{a_1, a_2, a_3}$ be an
optimal set of three-means such that $a_1<a_2<a_3$. Then, $a_1= \frac 16$, $a_2=\frac{13}{18}$, $a_3=\frac{17}{18}$, and  the corresponding quantization error is $V_3=\frac{5}{972}$.
\end{lemma}

\begin{lemma}  \label{lemma223}
Let $\ga:=\set{a_1, a_2, a_3, a_4}$ be an
optimal set of four-means such that $a_1<a_2<a_3<a_4$. Then, $a_1=\frac{1}{12}$, $a_2=\frac{1}{4}$, $a_3=\frac{13}{18}$, $a_4=\frac{17}{18}$, and  the corresponding quantization error is $V_4=\frac{13}{7776}$.
\end{lemma}

\begin{proof} Consider the set of four points $\gb:=\set {\frac{1}{12},\frac{1}{4},\frac{13}{18},\frac{17}{18}}$. The distortion error due to the set $\gb$ is given by
\begin{align*}
&\int\min_{a\in \gb}(x-a)^2 dP\\
&=\int_{[0, \frac 16]}(x-\frac 1{12})^2 dP+\int_{[\frac 16, \frac 13]}(x-\frac 14)^2 dP+\int_{J_2}(x-\frac{13}{18})^2 dP+\int_{J_3}(x-\frac {17}{18})^2 dP=\frac{13}{7776},
\end{align*}
implying $V_4\leq \frac{13}{7776}=0.00167181$.

Let $\ga:=\set{a_1<a_2<a_3<a_4}$ be an optimal set of four-means. Since optimal quantizers are the expected values of their own Voronoi regions, we have $0<a_1<a_2<a_3<a_4<1$. If $\frac 13\leq a_1$, then
\[V_4\geq \int_{J_1}(x-\frac 13)^2 dP= \frac{1}{54}=0.0185185>V_4,\]
which leads to a contradiction, so we can assume that $a_1<\frac 1 3$. Suppose that $\frac 13\leq a_2$. Then, the distortion error contributed by $a_1$ and $a_2$ on the set $J_1$ is given by
\begin{align*}
&\int_{[0, \frac 12(a_1+\frac 13)]}(x-a_1)^2 dP+\int_{[\frac 12(a_1+\frac 13), \frac 13]}(x-\frac 13)^2 dP=\frac{1}{216} \left(81 a_1^3+27 a_1^2-9 a_1+1\right),
\end{align*}
which is minimum when $a_1=\frac 19$, and the minimum value is $\frac{1}{486}=0.00205761>V_4$, which is a contradiction. So, we can assume that $0<a_1<a_2<\frac 13$. If $a_4\leq \frac 56$, then
\[V_4\geq \int_{J_3}(x-\frac 56)^2 dP=\frac{13}{3888}=0.00334362>V_4,\]
which leads to a contradiction. So, we can assume that $\frac 56<a_4$. Suppose that $a_3\leq \frac 12$. Then, $\frac 12(\frac 12+\frac 56)=\frac 23$ implying
\[V_4\geq \int_{J_2}(x-\frac 56)^2 dP=\frac{13}{3888}=0.00334362>V_4,\]
which is a contradiction. So, we can assume that $\frac 12<a_3$. Now, if the Voronoi region of $a_3$ contains points from $J_1$, we must have $\frac 12(a_2+a_3)<\frac 13$ implying $a_2<\frac 23-a_3\leq \frac 23-\frac 12=\frac 16$, and so,
\[V_4\geq \int_{[\frac 16, \frac 13]}(x-\frac 16)^2 dP=\frac{1}{432}=0.00231481>V_4,\]
which yields a contradiction. Thus, we can assume that the Voronoi region of $a_3$ does not contain any point from $J_1$ implying $\frac 23<a_3$. If $\frac 79\leq a_3$, then
\[V_4\geq V(P, \ga_2(P(\cdot|J_1), J_1), J_1)+\int_{J_2}(x-\frac 79)^2 dP=\frac{17}{7776}=0.00218621>V_4,\]
which gives a contradiction. So, we can assume that $\frac 23<a_3<\frac 79$. We now show that the Voronoi region of $a_4$ does not contain any point from $J_2$. If it does, then
\begin{align*}
\int_{[\frac 23, \frac 12(a_3+\frac 56)]}(x-a_3)^2 dP+\int_{[\frac 12(a_3+\frac 56), \frac 23]}(x-\frac 56)^2 dP=\frac{9 a_3^3}{16}-\frac{33 a_3^2}{32}+\frac{39 a_3}{64}-\frac{3541}{31104},
\end{align*}
which is minimum if $a_3=\frac {13}{18}$. Notice that $\frac 12(\frac {13}{18}+\frac 56)=\frac 79$ yielding the fact that $P$-almost surely the Voronoi region of $a_4$ does not contain any point from $J_2$ implying $\frac 89<a_4$. Thus, we see that
$a_1=\frac{1}{12}$, $a_2=\frac{1}{4}$, $a_3=\frac{13}{18}$ and $a_4=\frac{17}{18}$ and the corresponding quantization error is given by $V_4=\frac{13}{7776}$, which completes the proof of the lemma.
\end{proof}

\begin{prop}  \label{prop2122}
Let $n\geq 3$ and let $\ga_n$ be an optimal set of $n$-means. Then,

 $(i)$ $\ga_n\ii J_i\neq \es$ for all $1\leq i\leq 3$;

 $(ii)$ $\ga_n$ does not contain any point from the open intervals $(\frac 13, \frac 23)$ and $(\frac 79, \frac 89)$;

 $(iii)$ the Voronoi region of any point in $\ga_n\ii J_i$ does not contain any point from $J_j$ for $1\leq i\neq j\leq 3$.
\end{prop}

\begin{proof} From Lemma~\ref{lemma222b} and Lemma~\ref{lemma223}, it follows that the proposition is true for $n=3, 4$. We now prove that the proposition is true for $n\geq 5$. Consider the set of five points $\gb:=\set{\frac{1}{18},\frac{1}{6},\frac{5}{18},\frac{13}{18},\frac{17}{18}}$. The distortion error due to the set $\gb$ is given by
\begin{align*}
&\int\min_{a\in \gb}(x-a)^2 dP=\int_{J_1}\min_{a\in \set{\frac{1}{18},\frac{1}{6},\frac{5}{18}}}(x-a)^2 dP+\int_{J_2}(x-\frac{13}{18})^2 dP+\int_{J_3}(x-\frac {17}{18})^2 dP=\frac{1}{972},
\end{align*}
implying $V_5\leq \frac{1}{972}=0.00102881$. Since $V_n$ is the quantization error for $n$-means for all $n\geq 5$, we have $V_n\leq V_5\leq 0.00102881$. Let $\ga:=\set{a_1<a_2<a_3<a_4<a_5}$ be an optimal set of five-means. Since optimal quantizers are the expected values of their own Voronoi regions, we have $0<a_1<a_2<a_3<a_4<a_5<1$. If $\frac 13\leq a_1$, then
\[V_n\geq \int_{J_1}(x-\frac 13)^2 dP= \frac{1}{54}=0.0185185>V_n,\]
which leads to a contradiction, so we can assume that $a_1<\frac 1 3$, i.e., $\ga_n\ii J_1\neq \es$.  If $a_n\leq \frac 89$, then
\[V_n\geq \int_{J_2}\min_{a\in \ga}(x-a)^2 dP+\int_{J_3}(x-\frac 89)^2 dP>\int_{J_3}(x-\frac 89)^2 dP=\frac{1}{972}\geq V_n,\]
which is a contradiction. So, $\frac 89<a_n$ yielding $\ga_n\ii J_3\neq \es$. Let $j=\max\set {i : a_i<\frac 23}$. Then, $a_j<\frac 23$. We now show that $\ga_n$ does not contain any point from the open interval $(\frac 13, \frac 23)$.  For the sake of contradiction assume that $\ga_n$ contain a point from the open interval $(\frac 13, \frac 23)$. The following two cases can arise:

Case~1. $\frac 12\leq a_j<\frac 23$.

Then, $\frac 12(a_{j-1}+a_j)<\frac 13$ implying $a_{j-1}<\frac 23-a_j\leq \frac 23-\frac 12=\frac 16$, and so,
\[V_n\geq \int_{[\frac 16, \frac 13]}(x-\frac 16)^2 dP=\frac{1}{432}=0.00231481>V_n,\]
which is a contradiction.

Case~2. $\frac 13<a_j\leq \frac 12$.

Then, $\frac 12(a_{j}+a_{j+1})>\frac 23$ implying $a_{j+1}>\frac 43-a_j\geq \frac 4 3-\frac 12=\frac 56>\frac 79$, and so,
\[V_n\geq \int_{J_2}(x-\frac 56)^2 dP=\frac{13}{3888}=0.00334362>V_n,\]
which leads to a contradiction.

By Case~1 and Case~2, we can assume that $\ga_n$ does not contain any point from the open interval $(\frac 13, \frac 23)$. If $\frac 79\leq a_{j+1}$, then
\[V_n\geq \int_{J_1}\min_{a\in \ga_n}(x-a)^2 dP+ \int_{J_2}(x-\frac 79)^2 dP>\int_{J_2}(x-\frac 79)^2 dP=\frac{1}{972}\geq V_n,\]
which is a contradiction. So, we can assume that $a_{j+1}<\frac 79$ implying $\ga_n\ii J_2\neq \es$. If the Voronoi region of any point in $\ga_n\ii J_2$ contains points from $J_1$, then we must have $\frac 12(a_{j}+a_{j+1})<\frac 13$ implying $a_{j}<\frac 23-a_{j+1}\leq \frac 23-\frac 23=0$, which is a contradiction. If the Voronoi region of any point in $\ga_n\ii J_1$ contains points from $J_2$, then we must have $\frac 12(a_{j}+a_{j+1})>\frac 23$ implying $a_{j+1}>\frac 43-a_{j}\geq \frac 43-\frac 13=1$, which gives another contradiction. Hence, the Voronoi region of any point in $\ga_n\ii J_2$ does not contain any point from $J_1$, and the Voronoi region of any point in $\ga_n\ii J_1$ does not contain any point from $J_2$.

We now show that $\ga_n$ does not contain any point from the open interval $(\frac 79, \frac 89)$. Since $\ga_n$ does not contain any point from $(\frac 13, \frac 23)$ and the Voronoi region of any point in $\ga_n\ii J_2$ does not contain any point from $J_1$, and the Voronoi region of any point in $\ga_n\ii J_1$ does not contain any point from $J_2$, we have
\[\int_{[\frac 23, 1]} \min_{a\in \ga_n} (x-a)^2 dP=\int_{[\frac 23, 1]} \min_{a\in \ga_n\ii [\frac 23, 1]} (x-a)^2 dP.\]
Let $V(P, \ga_n\ii [\frac 23,1])$ be the quantization error contributed by the set $\ga_n\ii [\frac 23, 1]$ in the region $[\frac 23,1]$.
Since $\ga_n\ii J_2\neq \es$ and $\ga_n\ii J_3\neq \es$, if $\te{card}(\ga_n\ii [\frac 23, 1])=2$, then $\ga_n$ does not contain any point from  $(\frac 79, \frac 89)$. Assume that $\te{card}(\ga_n\ii [\frac 23, 1])=3$. Consider the set of three points $\gg=\set{\frac{25}{36},\frac{3}{4},\frac{17}{18}}$. Since,
\[\int_{[\frac 23, 1]}\min_{a\in \gg}(x-a)^2 dP=\int_{[\frac 23, \frac{13}{18}]}(x-\frac{25}{36})^2 dP+\int_{[\frac{13}{18}, \frac 79]}(x-\frac{3}{4})^2 dP+\int_{J_3}(x-\frac {17}{18})^2 dP=\frac{5}{15552},\]
we have $V(P, \ga_n\ii [\frac 23,1])\leq \frac{5}{15552}=0.000321502.$
If $\ga_n$ contains a point from $(\frac 79, \frac 89)$, we must have $\frac 79<a_{n-1}<\frac 89$. Suppose that $\frac 56\leq a_{n-1}<\frac 89$.
Then, $\frac 12(a_{n-2}+a_{n-1})<\frac 79$ implying $a_{n-2}<\frac {14}{9}-a_{n-1}\leq \frac{14}{9}-\frac 56=\frac{13}{18}$. Now, notice that
\begin{align*} \int_{J_2}\min_{a\in\ga_n \ii [\frac 23, 1]}(x-a)^2 dP
&=\int_{[\frac 23, \frac{1}{2} (a_{n-2}+\frac{5}{6})]}(x-a_{n-2})^2dP +\int_{[\frac{1}{2} (a_{n-2}+\frac{5}{6}), \frac 79]}(x-\frac 56)^2dP\\
&=\frac{9 a_{n-2}^3}{16}-\frac{33 a_{n-2}^2}{32}+\frac{39 a_{n-2}}{64}-\frac{3541}{31104},
\end{align*}
which is minimum if $a_{n-2}=\frac{13}{18}$, and then $\frac 12(a_{n-2}+a_{n-1})\geq \frac 12(\frac{13}{18}+\frac 56)=\frac 79$, which contradicts the fact that $\frac 12(a_{n-2}+a_{n-1})<\frac 79$. So, we can assume that $\frac 56\leq a_{n-1}<\frac 89$ is not true. Reflecting the situation with respect to the point $\frac 56$, we can show that $\frac 79< a_{n-1}\leq \frac 56$ is also not true. Therefore, if  $\te{card}(\ga_n\ii [\frac 23, 1])=3$, the set $\ga_n$ does not contain any point from $(\frac 79, \frac 89)$. Next, assume that  $\te{card}(\ga_n\ii [\frac 23, 1])= m$ for some positive integer $m\geq 4$. Let $k=\max\set{i : a_i<\frac 89}$. Then, $a_k<\frac 89$. We need to show that $a_k\leq \frac 79$. Consider the set of four points $\gd:=\set{\frac{25}{36},\frac{3}{4},\frac{11}{12},\frac{35}{36}}$. Since $V(P, \ga_n\ii [\frac 23, 1])$ is the quantization error for $m$-means for $m\geq 4$, we have
\[V(P, \ga_n\ii [\frac 23, 1])\leq \int_{[\frac 23, 1]}\min_{a\in\gd}(x-a)^2 dP=\frac{1}{7776}=0.000128601.\]
For the sake of contradiction, assume that $\frac 79<a_k<\frac 89$. The following two cases can arise:

Case~A. $\frac 56\leq a_k< \frac 89$.

Then, $\frac 12(a_{k-1}+a_k)<\frac 79$ implying $a_{k-1}<\frac {14}9-a_k=\frac {14}9-\frac 56=\frac{13}{18}$, and so,
\[V(P, \ga_n\ii [\frac 23, 1])\geq \int_{[\frac {13}{18}, \frac 79]}(x-\frac {13}{18})^2 dP+\int_{J_2}\min_{a\in \ga_n}(x-a)^2 dP>\int_{[\frac {13}{18}, \frac 79]}(x-\frac {13}{18})^2 dP=\frac{1}{7776},\]
implying $V(P, \ga_n\ii [\frac 23, 1])>\frac{1}{7776}=V(P, \ga_n\ii [\frac 23, 1])$,
which is a contradiction.

Case~B. $\frac 79 < a_k\leq \frac 56$.

Reflecting the situation in Case~A with respect to the point $\frac 56$, in this case, we can also show that a contradiction arises.

Hence, by Case~A and Case~B, we can assume that $\ga_n$ does not contain any point from the open interval $(\frac 79, \frac 89)$, i.e., $a_k\leq \frac 79$. If the Voronoi region of any point in $\ga_n\ii J_3$ contains points from $J_2$, then we must have $\frac 12(a_{k}+a_{k+1})<\frac 79 $ implying $a_{k}<\frac {14}9 -a_{k+1}\leq \frac {14}9-\frac 89=\frac 23$, which contradicts the fact that $\ga_n\ii J_2\neq\es$. If the Voronoi region of any point in $\ga_n\ii J_2$ contains points from $J_3$, then we must have $\frac 12(a_{k}+a_{k+1})>\frac 89$ implying $a_{k+1}>\frac {16}{9}-a_{k}\geq \frac {16}{9}-\frac 79=1$, which gives another contradiction. Hence, the Voronoi region of any point in $\ga_n\ii J_3$ does not contain any point from $J_2$, and the Voronoi region of any point in $\ga_n\ii J_2$ does not contain any point from $J_3$. Thus, the proof of the proposition is complete.
\end{proof}

Due to Proposition~\ref{prop2122}, we are now ready to state and prove the following proposition, which helps us to determine the optimal sets of $n$-means and the $n$th quantization errors for all $n\geq 3$ as stated in the subsequent notes.
\begin{prop}\label{prop2155}
Let $\ga_n$ be an optimal set of $n$-means for $n\geq 3$. Write $\ga_{n,j}:=\ga_n\ii J_j$ and $n_j:=\te{card}(\ga_{n,j})$ for $1\leq j\leq 3$. Then, $\ga_{n, j}=\ga_{n_j}(P(\cdot|J_j))$ and $n=n_1+n_2+n_3$, with
\begin{equation} \label{eq89}
V_n=\sum_{j=1}^3 V(P, \ga_{n_j}(P(\cdot|J_j)), J_j)=\frac 1{216} \frac 1{n_1^2}+\frac 1{3888}\Big(\frac 1{n_2^2}+\frac 1{n_3^2}\Big).
\end{equation}
\end{prop}

\begin{proof}  If $\ga_{n, j}$ is not an optimal set of $n_j$-means with respect to the probability distribution $P(\cdot|J_j)$, we must have another set $\ga_{n, j}'$ with cardinality $n_j$ which will give smaller distortion error with respect to $P(\cdot|J_j)$ than the distortion error due to the set $\ga_{n, j}$. This will contradict the fact that $\ga_n$ is an optimal set of $n$-means with respect to the probability distribution $P$. Since $\ga_{n,j}$ are disjoint for $1\leq j\leq 3$ and $\ga_n$ does not contain any point from the open intervals $(\frac 13, \frac 23)$ and $(\frac 79, \frac 89)$, we have
$\ga_n=\ga_{n, 1}\uu \ga_{n, 2}\uu \ga_{n, 3}$ and $n=n_1+n_2+n_3$, and so,
\begin{align*} V_n&=\int \min_{a \in \ga_{n}} (x-a)^2 dP=\sum_{j=1}^{3} \int_{J_j}\min_{a \in \ga_{n,j}} (x-a)^2 dP
=\sum_{j=1}^{3} V(P, \ga_{n_j}(P(\cdot|J_j)), J_j)\\
&=\frac 1{216} \frac 1{n_1^2}+\frac 1{3888}\Big(\frac 1{n_2^2}+\frac 1{n_3^2}\Big).
\end{align*}
Thus, the proof of the proposition is complete.
\end{proof}

\begin{note} Since $V_n$ represents the $n$th quantization error for any $n\in \D N$, if $n_2+n_3=m$ for some positive integer $m$, the expression $\frac 1{3888}\Big(\frac 1{n_2^2}+\frac 1{n_3^2}\Big)$ is minimum if $n_2\approx \frac m 2$ and $n_3\approx \frac m 2$. Thus, we see that if $m=2k$ for some positive integer $k$, then $n_2=n_3=k$, and if $m=2k+1$ for some positive integer $k$, then either $(n_2=k+1$ and $n_3=k)$ or $(n_2=k$ and $n_3=k+1)$. Moreover, writing $n_2=n_3$, or $n_2=n_3+1$ in \eqref{eq89}, it can be seen that $n_1\geq \frac n 2$ for any positive integer $n\geq 4$. Thus, we see that unlike the uniform distribution with infinitely many pieces, described in the previous section, the optimal sets of $n$-means for the uniform distribution with finitely many pieces for all $n\in \D N$ are not unique: if $n_2+n_3$ is an odd number then there are two different optimal sets of $n$-means, and if $n_2+n_3$ is an even number then the optimal set of $n$-means is unique.
\end{note}

In the following note we describe how to determine the optimal sets of $n$-means and the $n$th quantization errors for all $n\geq 3$.
\begin{note} \label{note12}
To determine an optimal set of $n$-means for any positive integer $n\geq 3$, we need to know $n_1$, $n_2$ and $n_3$ as described in Proposition~\ref{prop2155}. Notice that for any $n\in \D N$, $n\geq 3$, we can easily determine $n_1$, $n_2$ and $n_3$ by minimizing the following function:
\[f(n_1, n_2, n_3):=\frac 1{216} \frac 1{n_1^2}+\frac 1{3888}\Big(\frac 1{n_2^2}+\frac 1{n_3^2}\Big),\]
subject to the constraint $n_1+n_2+n_3=n$. Once $n_1, n_2$ and $n_3$ are known, then by Proposition~\ref{prop111}, using the following formula we can determine the corresponding optimal set of $n$-means:
\[\ga_n=\ga_{n_1}(P(\cdot|J_1))\uu \ga_{n_2}(P(\cdot|J_2))\uu \ga_{n_3}(P(\cdot|J_3)).\]
For example: If $n=7$, then $\set{n_1=4,\,  n_2=2, \, n_3=1}$, or $\set{n_1=4,\,  n_2=1, \, n_3=2}$ and the corresponding quantization error is $\frac{19}{31104}$. If $n=100$, then $\set{n_1=56, \, n_2=n_3=22}$ and the corresponding quantization error is $\frac{1873}{737662464}$, etc.

\end{note}

\noindent {\bf Acknowledgments} We thank B. Pittel for useful facts about random quantizations and suggestions for some possible asymptotic results.  We also would like to thank the referee whose questions and suggestions have been very important in improving this article, both in terms of its content and its citations.

\end{document}